\documentclass[10pt,a4paper]{article}
\usepackage{amsmath,amsfonts,amsthm}
\usepackage{a4wide}
\usepackage{authblk}
\usepackage{bbold}
\usepackage{booktabs}
\usepackage{enumerate}
\usepackage{general}
\usepackage{hyperref}
\usepackage[bold]{hhtensor}
\usepackage{multirow}
\usepackage{paralist}
\usepackage{xcolor}

\newcommand{\email}[1]{\href{mailto:#1}{#1}}

\numberwithin{equation}{section}

\newcommand{\Hdiv}[1][\Omega]{\vec{H}(\opdiv;#1)}
\newcommand{\Hrot}[1][\Omega]{\vec{H}(\oprot;#1)}
\newcommand{\Hcurl}[1][\Omega]{\vec{H}(\opcurl;#1)}

\newcommand{\lproj}[2][T]{\pi_{#1}^{#2}}
\newcommand{\vlproj}[2][T]{\vec{\pi}_{#1}^{#2}}

\newcommand{\cG}{\vec{\mathbb{G}}}
\newcommand{\coG}{\overline{\vec{\mathbb{G}}}}
\newcommand{\cGT}[1][l]{\cG^{#1}_T}
\newcommand{\coGT}[1][m]{\coG^{#1}_T}

\newcommand{\cS}{\vec{\mathbb{S}}}
\newcommand{\cST}[1][k,m]{\cS^{#1}_T}

\newcommand{\vsigma}{\vec{\sigma}}
\newcommand{\vtau}{\vec{\tau}}
\newcommand{\uvsigma}[1][T]{\underline{\vec{\sigma}}_{#1}}
\newcommand{\uvhsigma}[1][T]{\widehat{\underline{\vec{\sigma}}}_{#1}}
\newcommand{\uvtau}[1][T]{\underline{\vec{\tau}}_{#1}}
\newcommand{\vchi}{\vec{\chi}}
\newcommand{\uvchi}[1][T]{\underline{\vec{\chi}}_{#1}}
\newcommand{\uvhchi}[1][T]{\widehat{\underline{\vec{\chi}}}_{#1}}

\newcommand{\vSigma}{\vec{\Sigma}}
\newcommand{\uvSigmaT}[1][k,l,m]{\underline{\vSigma}_T^{#1}}
\newcommand{\uvIST}[1][k,l,m]{\underline{\vec{I}}_{\vSigma,T}^{#1}}
\newcommand{\uvISh}[1][k,l,m]{\underline{\vec{I}}_{\vSigma,h}^{#1}}

\newcommand{\uvcSigmah}[1][k,l,m]{\underline{\widecheck{\vSigma}}_h^{#1}}
\newcommand{\uvSigmah}[1][k,l,m]{\underline{\vSigma}_h^{#1}}

\newcommand{\uUTs}[1][k,l]{\underline{U}_{T,\asterisk}^{#1}}

\newcommand{\UT}[1][l]{U_T^{#1}}
\newcommand{\uUT}[1][k,l]{\underline{U}_T^{#1}}
\newcommand{\Uh}[1][l]{U_h^{#1}}
\newcommand{\uUh}[1][k,l]{\underline{U}_{h}^{#1}}
\newcommand{\uUhD}[1][k,l]{\underline{U}_{h,0}^{#1}}

\newcommand{\uIUT}[1][k,l]{\underline{I}_{U,T}^{#1}}
\newcommand{\uIUh}[1][k,l]{\underline{I}_{U,h}^{#1}}

\newcommand{\uu}[1][h]{\underline{u}_{#1}}
\newcommand{\uuh}[1][h]{\widehat{\underline{u}}_{#1}}
\newcommand{\hu}[1][h]{\widehat{u}_{#1}}
\newcommand{\uv}[1][h]{\underline{v}_{#1}}
\newcommand{\uw}[1][T]{\underline{w}_{#1}}
\newcommand{\cu}[1][T]{\widecheck{u}_{#1}}
\newcommand{\uhz}[1][h]{\widehat{\underline{z}}_{#1}}
\newcommand{\hz}[1][h]{\widehat{z}_{#1}}
\newcommand{\cz}[1][h]{\widecheck{z}_{#1}}

\newcommand{\vX}{\vec{X}}
\newcommand{\uvXh}[1][k,l,m]{\underline{\vec{X}}_h^{#1}}

\newcommand{\DT}[1][l]{\mathrm{D}_T^{#1}}
\newcommand{\Dh}[1][l]{\mathrm{D}_h^{#1}}
\newcommand{\PT}[1][k]{\vec{\mathrm{P}}_T^{#1}}
\newcommand{\ST}[1][k]{\vec{\mathrm{S}}_{T}^{#1}}

\newcommand{\vsT}[1][k,l,m]{\underline{\vec{\varsigma}}_{T}^{#1}}
\newcommand{\vsh}[1][k,l,m]{\underline{\vec{\varsigma}}_{h}^{#1}}

\newcommand{\GT}[1][k]{\vec{\mathrm{G}}_T^{#1}}

\newcommand{\pT}[1][k+1]{\mathrm{p}_T^{#1}}

\newcommand{\sprime}{\textsuperscript{\ensuremath{\prime}}}

\newcommand{\PTF}[1][TF]{\mathcal{P}_{#1}}
\newcommand{\vfS}[1]{\vec{\mathfrak{S}}^{#1}}
\newcommand{\deTF}[1][k]{\delta_{TF}^{#1}}

\newcommand{\RT}[1][k]{\mathbb{RT}^{#1}}
\newcommand{\NC}{\mathbb{NC}}

\title{Unified formulation and analysis of mixed and primal discontinuous skeletal methods on polytopal meshes\footnote{The work of D. Boffi was partially supported by PRIN/MIUR, by GNCS/INDAM,
and by IMATI/CNR. The work of D. A. Di Pietro was partially supported by Agence Nationale de la Recherche project HHOMM (ANR-15-CE40-0005).}}
\author[1]{Daniele Boffi\footnote{\email{daniele.boffi@unipv.it}}}
\affil[1]{Universit\`{a} degli Studi di Pavia, Dipartimento di Matematica ``Felice Casorati'', 27100 Pavia, Italy}
\author[2]{Daniele A. Di Pietro\footnote{\email{daniele.di-pietro@umontpellier.fr}}}
\affil[2]{Universit\'{e} de Montpellier, Institut Montpelli\'{e}rain Alexander Grothendieck, 34095 Montpellier, France}

\begin{document}

\maketitle

\begin{abstract}
  We propose in this work a unified formulation of mixed and primal discretization methods on polyhedral meshes hinging on globally coupled degrees of freedom that are discontinuous polynomials on the mesh skeleton. To emphasize this feature, these methods are referred to here as discontinuous skeletal.
  As a starting point, we define two families of discretizations corresponding, respectively, to mixed and primal formulations of discontinuous skeletal methods.
  Each family is uniquely identified by prescribing three polynomial degrees defining the degrees of freedom and a stabilization bilinear form which has to satisfy two properties of simple verification: stability and polynomial consistency.
  Several examples of methods available in the recent literature are shown to belong to either one of those families.
  We then prove new equivalence results that build a bridge between the two families of methods.
  Precisely, we show that for any mixed method there exists a corresponding equivalent primal method, and the converse is true provided that the gradients are approximated in suitable spaces.
  A unified convergence analysis is also carried out delivering optimal error estimates in both energy- and $L^2$-norm.
\end{abstract}



\section{Introduction}

Over the last few years, discretization methods that support general polytopal meshes have received a great amount of attention.
Such methods are often formulated in terms of two sets of degrees of freedom (DOFs) located inside mesh elements and on the mesh skeleton, respectively.
The former can often be eliminated (possibly after hybridization) by static condensation, whereas the latter are responsible for the transmission of information among elements, and are therefore globally coupled.
To emphasize the role of the second set of DOFs, these methods are referred to here as ``skeletal''.
Skeletal methods can be classified according to the continuity property of skeletal DOFs on the mesh skeleton.
We focus here on ``discontinuous skeletal'' methods, where skeletal DOFs are single-valued polynomials over faces fully discontinuous at the face boundaries.
Since this terminology is not classical in the sense of standard finite elements, we explicitly point out that here single-valued means that interface values match from one element to the adjacent one.
Discontinuous, on the other hand, refers to the fact that skeletal DOFs are discontinuous at vertices in 2d and edges in 3d.

Let $\Omega\subset \Real^d$, $d\ge1$, denote an open, bounded, connected polytopal set, and let $f\in L^2(\Omega)$.
To avoid unnecessary complications, we consider the following pure diffusion model problem:
Find $u:\Omega\to\Real$ such that
\begin{equation}
  \label{eq:strong}
  \begin{alignedat}{2}
    -\LAPL u &= f &\qquad&\text{in $\Omega$}, \\
    u &= 0 &\qquad&\text{on $\partial\Omega$}.
  \end{alignedat}
\end{equation}
We introduce a unified formulation of discontinuous skeletal discretizations of problem~\eqref{eq:strong} which encompasses a large number of schemes from the literature.
As a starting point, we define two families of discretizations corresponding, respectively, to mixed and primal discontinuous skeletal methods.
Each family is uniquely identified by prescribing three polynomial degrees defining element-based and skeletal DOFs, and a stabilization bilinear form which has to satisfy two properties of simple verification: stability expressed in terms of a uniform norm equivalence, and polynomial consistency.
Several examples of methods available in the recent literature are shown to belong to either one of those families.
We then prove new equivalence results, collected in Theorems~\ref{thm:hybridization},~\ref{thm:equivalence:m->p}, and~\ref{thm:equivalence:p->m} below,  which build a bridge between the two families of methods.
Precisely, we show that for any mixed method there exists a corresponding equivalent primal method, and the converse is true provided that the gradients are approximated in suitable spaces.
A unified convergence analysis is also carried out delivering optimal error estimates in both energy- and $L^2$-norms; cf. Theorems~\ref{thm:en.err.est} and~\ref{thm:l2.err.est} below.

A fundamental and inspiring example is presented in
Section~\ref{sec:RT0.CR0}: it refers to the well-known equivalence between
the lowest-order Raviart--Thomas element and the nonconforming Crouzeix--Raviart
element on triangular meshes. In some sense, the framework presented in this
paper extends, with suitable modifications, this equivalence to recent methods
supporting general polytopal meshes.

Polytopal methods were first investigated in the context of lowest-order discretizations starting from several different points of view.
In the context of finite volume schemes, several families of polyhedral methods have been developed as an effort to weaken the conditions on the mesh required for the consistency of classical five-point schemes.
The resulting methods are expressed in terms of local balances, and an explicit expression for the numerical fluxes is usually available.
Discontinuous skeletal methods in this context include the Mixed and Hybrid Finite Volume schemes of~\cite{Droniou.Eymard:06,Eymard.Gallouet.ea:10}.
Continuous skeletal methods have also been considered, e.g., in~\cite{Eymard.Guichard.ea:12}.

Relevant features of the continuous problem different from local conservation have inspired other approaches.
Mimetic Finite Difference methods are derived by using discrete integration by parts formulas to define the counterparts of differential operators and $L^2$-products; cf.~\cite{Beirao-da-Veiga.Lipnikov.ea:14} for an introduction.
Discontinuous skeletal methods in this context include, in particular, the mixed Mimetic Finite Difference scheme of~\cite{Brezzi.Lipnikov.ea:05}.
An example of continuous skeletal method is provided, on the other hand, by the nodal scheme of~\cite{Brezzi.Buffa.ea:09}.
In the Discrete Geometric Approach~\cite{Codecasa.Specogna.ea:10}, the formal links with the continuous operators are expressed in terms of Tonti diagrams~\cite{Tonti:75}.
We also cite in this context the Compatible Discrete Operator framework of~\cite{Bonelle.Ern:14}.
To different extents, all of the previous methods can be linked to the seminal ideas of Whitney on geometric integration.
Other methods that deserve to be cited here are the cell centered Galerkin methods of~\cite{Di-Pietro:12,Di-Pietro:13}, which can be regarded as discontinuous Galerkin methods with only one unknown per element where consistency is achieved by the use of cleverly-tailored reconstructions.

The close relation among the Mixed~\cite{Droniou.Eymard:06} and Hybrid~\cite{Eymard.Gallouet.ea:10} Finite Volume schemes and mixed Mimetic Finite Difference methods~\cite{Brezzi.Lipnikov.ea:05} has been investigated in~\cite{Droniou.Eymard.ea:10}, where equivalence at the algebraic level is demonstrated for generalized versions of such schemes; cf. also~\cite[Section~7]{Vohralik.Wohlmuth:13} for further insight into the link with submesh-based polyhedral implementations of classical mixed finite elements.
The results of~\cite{Droniou.Eymard.ea:10} are recovered here as a special case.
A unifying point of view for the convergence analysis has been recently proposed in~\cite{Droniou.Eymard.ea:13} under the name of Gradient Schemes.
Finally, the methods discussed above can often be regarded as lowest-order versions of more recent polytopal technologies such as, e.g., Virtual Elements and Hybrid High-Order methods.

A natural development of polytopal methods was headed to increase the approximation order.
It has been known for quite some time that high-order polyhedral discretizations can be obtained by fully nonconforming approaches such as the discontinuous Galerkin method.
An exposition of the basic analysis tools in this framework can be found in~\cite{Di-Pietro.Ern:12}; cf. also~\cite{Di-Pietro.Droniou:15,Di-Pietro.Droniou:16} for polynomial approximation results on polyhedral elements based on the Dupont-Scott theory~\cite{Dupont.Scott:80} and~\cite{Bassi.Botti.ea:12,Antonietti.Giani.ea:13,Cangiani.Georgoulis.ea:14} for further developments.
Particularly interesting among discontinuous Galerkin methods is the hybridizable version introduced in~\cite{Castillo.Cockburn.ea:00,Cockburn.Gopalakrishnan.ea:09}, which constitutes a first example of high-order discontinuous skeletal method.

Very recent works have shown other possible approaches to the design of high-order polytopal discretizations combining element-based and skeletal unknowns.
A first example of arbitrary-order discontinuous skeletal methods are primal~\cite{Di-Pietro.Ern.ea:14,Di-Pietro.Droniou.ea:15} and mixed~\cite{Di-Pietro.Ern:16} Hybrid High-Order methods.
Hybrid High-Order methods were originally introduced in~\cite{Di-Pietro.Ern:15*1} in the context of linear elasticity.
The main idea consists in reconstructing high-order differential operators based on suitably selected DOFs and discrete integration by parts formulas.
These reconstructions are then used to formulate the local contributions to the discrete problem including a cleverly tailored stabilization that penalizes high-order face-based residuals.
A study of the relations among primal Hybrid High-Order methods, Hybridizable Discontinuous Galerkin (HDG) methods, and High-Order Mimetic Finite Differences~\cite{Lipnikov.Manzini:14} can be found in~\cite{Cockburn.Di-Pietro.ea:15}, where the corresponding numerical fluxes in the spirit of HDG methods are identified.
The hybridization of the original mixed Hybrid High-Order method was studied in~\cite{Aghili.Boyaval.ea:15} (these results are recovered as a special case in this work).

Another framework including both continuous and discontinuous skeletal methods is provided by Virtual Elements~\cite{Beirao-da-Veiga.Brezzi.ea:13,Beirao-da-Veiga.Brezzi.ea:13*1}.
Virtual Elements can be described as finite elements where the expressions of the basis functions are not available at each point, but suitable projections thereof can be computed using the selected DOFs.
Such computable projections are then used to approximate bilinear forms, which also include a stabilization term that penalizes differences between the DOFs and the computable projection.
We are particularly interested here in mixed~\cite{Brezzi.Falk.ea:14,Beirao-da-Veiga.Brezzi.ea:15,Beirao-da-Veiga.Brezzi.ea:16} and nonconforming~\cite{Ayuso-de-Dios.Lipnikov.ea:15} Virtual Elements, both of which are discontinuous skeletal methods.

The rest of this paper is organized as follows.
In Section~\ref{sec:mesh} we formulate the assumptions on the mesh and introduce the main notation.
In Section~\ref{sec:RT0.CR0} we recall the classical equivalence of lowest-order Raviart--Thomas and nonconforming finite element methods.
In Sections~\ref{sec:mixed} and~\ref{sec:primal} we introduce the families of mixed and primal discontinuous skeletal methods under study, and provide several examples of lowest-order and high-order methods that fall in each category.
In Section~\ref{sec:m->p} we show how to obtain, starting from a discontinuous skeletal method in mixed formulation, an equivalent primal method.
Conversely, in Section~\ref{sec:p->m}, we show how to derive an equivalent mixed formulation starting from a discontinuous skeletal method in primal formulation.
Section~\ref{sec:analysis} contains a unified convergence analysis yielding optimal error estimates in the energy- and $L^2$-norms.


\section{Mesh and notation}\label{sec:mesh}

Let ${\cal H}\subset \Real_*^+ $ denote a countable set of meshsizes having $0$ as its unique accumulation point.
We consider refined mesh sequences $(\Th)_{h \in {\cal H}}$ where, for all $ h \in {\cal H} $, $\Th = \{T\}$ is a finite collection of nonempty disjoint open polytopal elements such that $\overline{\Omega}=\bigcup_{T\in\Th}\closure{T}$ and $h=\max_{T\in\Th} h_T$
($h_T$ stands for the diameter of $T$).
For $X\subset\Real^d$, we denote by $\meas[N]{X}$ the $N$-dimensional Hausdorff measure of $X$.
A hyperplanar closed connected subset $F$ of $\closure{\Omega}$ is called a face if $\meas[d-1]{F}>0$ and
\begin{inparaenum}[(i)]
\item either there exist distinct $T_1,T_2\in\Th $ such that $F=\partial T_1\cap\partial T_2$ (and $F$ is an interface) or 
\item there exists $T\in\Th$ such that $F=\partial T\cap\partial\Omega$ (and $F$ is a boundary face).
\end{inparaenum}
The set of interfaces is denoted by $\Fhi$, the set of boundary faces by $\Fhb$, and we let
$\Fh\eqbydef\Fhi\cup\Fhb$.
For all $T\in\Th$, the sets $\Fh[T]\eqbydef\{F\in\Fh\st F\subset\partial T\}$ and $\Fhi[T]\eqbydef\Fh[T]\cap\Fhi$ collect, respectively, the faces and interfaces lying on the boundary of $T$ and, for all $F\in\Fh[T]$, we denote by $\normal_{TF}$ the normal to $F$ pointing out of $T$.
Symmetrically, for all $F\in\Fh$, $\Th[F]\eqbydef\{T\in\Th\st F\subset\partial T\}$ is the set containing the one or two elements sharing $F$.

We assume that $(\Th)_{h\in{\cal H}}$ is admissible in the sense of~\cite[Chapter~1]{Di-Pietro.Ern:12}, i.e., for all $h\in{\cal H}$, $\Th$ admits a matching simplicial submesh $\fTh$ and there exists a real number $\varrho>0$ (the mesh regularity parameter) independent of $h$ such that the following conditions hold:
\begin{inparaenum}[(i)]
\item For all $h\in{\cal H}$ and all simplex $S\in\fTh$ of diameter $h_S$ and inradius $r_S$, $\varrho h_S\le r_S$;
\item for all $h\in{\cal H}$, all $T\in\Th$, and all $S\in\fTh$ such that $S\subset T$, $\varrho h_T \le h_S$.
\end{inparaenum}
We refer to~\cite[Chapter~1]{Di-Pietro.Ern:12} and~\cite{Di-Pietro.Droniou:15,Di-Pietro.Droniou:16} for a set of geometric and functional analytic results valid on admissible meshes.

Let $X$ be a mesh element or face.
For an integer $l\ge 0$, we denote by $\Poly{l}(X)$ the space spanned by the restriction to $X$ of $d$-variate polynomials of total degree $l$.
We denote by $(\cdot,\cdot)_X$ and $\norm[X]{{\cdot}}$ the usual inner product and norm of $L^2(X)$.
The index is dropped when $X=\Omega$.
The $L^2$-projector $\lproj[X]{l}:L^1(X)\to\Poly{l}(X)$ is defined such that, for all $v\in L^1(X)$,
\begin{equation}\label{eq:lproj}
  (\lproj[X]{l}v-v,w)_X=0\qquad\forall w\in\Poly{l}(X).
\end{equation}

Let a mesh element $T\in\Th$ be fixed.
For all integer $l\ge 0$ we set
$$
\cGT[l]\eqbydef\GRAD\Poly{l+1}(T),\qquad
\coGT[l]\eqbydef\left\{
\vtau\in\Poly{l}(T)^d\st (\vtau,\GRAD w)_T=0\quad\forall w\in\Poly{l+1}(T)
\right\},
$$
and denote by $\vlproj[\cG,T]{l}:L^1(T)^d\to\cGT[l]$ and $\vlproj[\coG,T]{l}:L^1(T)^d\to\coGT[l]$ the $L^2$-orthogonal projectors on $\cGT[l]$ and $\coGT[l]$, respectively.
Clearly, we have the direct decomposition
\begin{equation}\label{eq:decomp.PTs}
  \Poly{l}(T)^d = \cGT[l] \oplus \coGT[l].
\end{equation}
For further use, at the global level, we also define the space of broken polynomials
$$
\Poly{l}(\Th)\eqbydef\left\{ v_h\in L^2(\Omega)\st v_T\eqbydef\restrto{v_h}{T}\in\Poly{l}(T)\quad\forall T\in\Th \right\}.
$$

Throughout the paper, to avoid naming constants, we use the abridged notation $a\lesssim b$ for the inequality $a\le Cb$ with real number $C>0$ independent of $h$.
We will also write $a\approx b$ to mean $a\lesssim b\lesssim a$.


\section{An inspiring example}\label{sec:RT0.CR0}

In order to put the following discussion into perspective, we start by recalling an important inspiring example, viz. the well-known equivalence between lowest-order
Raviart--Thomas element and nonconforming Crouzeix--Raviart element on
triangular meshes.

The Raviart--Thomas element~\cite{Raviart.Thomas:77} is widely used for the
approximation of problems involving $\Hdiv$ when $\Th$ is a matching triangular mesh. A popular implementation of
the Raviart--Thomas scheme makes use of a hybridization procedure, introducing a
Lagrange multiplier in order to enforce the continuity of the normal component
of vectors from one element to the other.
As a starting point, problem~\eqref{eq:strong} is written
in mixed form as follows: Find the flux $\vsigma\in\Hdiv$ and the potential $u\in L^2(\Omega)$ such
that
\[
\begin{alignedat}{2}
  (\vsigma,\vtau)+(\opdiv\vtau,u)&=0&\qquad&\forall\vtau\in\Hdiv,
  \\
  -(\opdiv\vsigma,v)&=(f,v)&\qquad&\forall v\in L^2(\Omega).
\end{alignedat}
\]
Taking the Raviart--Thomas finite element space $\RT[0](\Th)\subset\Hdiv$ for the flux and the space of piecewise constants $\Poly{0}(\Th)\subset L^2(\Omega)$ for the potential,
its discretization reads: Find $\vsigma_h\in\RT[0](\Th)$ and $u_h\in\Poly{0}(\Th)$ such
that
\begin{equation}
\begin{alignedat}{2}
  (\vsigma_h,\vtau_h)+(\opdiv\vtau_h,u_h)&=0&\qquad&\forall\vtau_h\in\RT[0](\Th),
  \\
  -(\opdiv\vsigma_h,v_h)&=(f,v_h)&\qquad&\forall v_h\in\Poly{0}(\Th).
\end{alignedat}
\label{eq:RT}
\end{equation}
The hybridized version of~\eqref{eq:RT} consists in introducing the space
$\Lambda_h$ of piecewise constants on the internal portion of the mesh skeleton, and in
solving the following problem which involves the \emph{discontinuous}
Raviart--Thomas space $\RT[0,\mathrm{d}](\Th)$:
Find $\vsigma_h\in\RT[0,\mathrm{d}](\Th)$, $u_h\in\Poly{0}(\Th)$, and $\lambda_h\in\Lambda_h$ such that
\begin{equation}
\begin{alignedat}{2}
  (\vsigma_h,\vtau_h)+(\opdiv\vtau_h,u_h)
  +\sum_{T\in\Th}\sum_{F\in\Fhi[T]}(\vtau_h\SCAL\normal_{TF},\lambda_h)_F
  &=0&\qquad&\forall\vtau_h\in\RT[0,\mathrm{d}](\Th),
  \\
  -(\opdiv\vsigma_h,v_h)&=(f,v_h)&\qquad&\forall v_h\in\Poly{0}(\Th),
  \\
  \sum_{T\in\Th}\sum_{F\in\Fhi[T]}(\vsigma_h\SCAL\normal_{TF},\mu_h)_F&=0&\qquad&\forall\mu_h\in\Lambda_h.
\end{alignedat}
\label{eq:RThyb}
\end{equation}
The usual way of solving problem~\eqref{eq:RThyb} is to invert the
(block-diagonal) mass matrix corresponding to the variables in $\RT[0,\mathrm{d}](\Th)$ and
to consider a statically condensed linear system of the form
\begin{equation}
\mathbf{A}\mathsf{\Lambda}=\mathsf{F}
\label{eq:RTstatcond}
\end{equation}
where $\mathbf{A}$ is symmetric and positive definite.

Let now $\NC(\Th)$ be the nonconforming Crouzeix--Raviart space
of~\cite{Crouzeix.Raviart:73} on the same mesh $\Th$; i.e., the space of
piecewise affine functions which are continuous on the midnodes of the
interelement edges.
Denoting by $\NC_0(\Th)$ the subspace of $\NC(\Th)$ with DOFs lying on $\partial\Omega$ set to zero, the approximation of problem~\eqref{eq:strong} reads: Find
$u_h\in\NC_0(\Th)$ such that
\begin{equation}
(\GRADh u_h,\GRADh v_h)=(f,v_h)\qquad\forall v_h\in\NC_0(\Th),
\label{eq:nonc}
\end{equation}
where $\GRADh$ denotes the broken gradient operator on $\Th$.
The matrix form of~\eqref{eq:nonc} is
\[
\mathbf{B}\mathsf{U}=\mathsf{G}
\]
with $\mathbf{B}$ symmetric and positive definite.
It is now well understood that the matrices $\mathbf{A}$ and
$\mathbf{B}$ are identical, as well as the corresponding right hand
sides $\mathsf{F}$ and $\mathsf{G}$. This important equivalence is a
consequence of the results of~\cite{Arnold.Brezzi:85,
Marini:85},~\cite{Arbogast.Chen:95, Chen:96}, and has been reported in this
form in~\cite{Vohralik.Wohlmuth:13}.

A natural question is whether results of this type can be obtained for higher
order schemes on general polytopal meshes. The results that we are going to
present aim at describing a unified setting where the equivalence of primal,
mixed, and hybrid formulation can be proved.
For a discussion of lowest-order Raviart--Thomas and Crouzeix--Raviart elements in the framework
introduced in the following section, we refer to Examples~\ref{ex:RT} and~\ref{ex:NC}, respectively.


\section{A family of mixed discontinuous skeletal methods}\label{sec:mixed}

In this section we introduce a family of mixed discontinuous skeletal methods and provide a few examples of members of this family.

\subsection{Local spaces}

For a given integer $k\ge 0$ corresponding to the skeletal polynomial degree, we let $l$ and $m$ be two integers such that
\begin{equation}\label{eq:l.m}
  \max(0,k-1)\le l\le k+1,\qquad
  m\in\{0,k\}.
\end{equation}
Let a mesh element $T\in\Th$ be given.
We define the following space of flux degrees of freedom (DOFs):
\begin{equation}\label{eq:uvSigmaT}
  \uvSigmaT\eqbydef(\cGT[l-1]\oplus\coGT)\times\left(\bigtimes_{F\in\Fh[T]}\Poly{k}(F)\right).
\end{equation}
For a generic element $\uvtau$ of $\uvSigmaT$ we use the notation $\uvtau=(\vtau_T,(\tau_{TF})_{F\in\Fh[T]})$ with $\vtau_T=\vtau_{\cG,T}+\vtau_{\coG,T}$.
For a fixed Lebesgue index $s>2$, we let $\vSigma^+(T)\eqbydef\{\vtau\in L^s(T)^d\st\DIV\vtau\in L^2(T)\}$ and define the local flux reduction map $\uvIST:\vSigma^+(T)\to\uvSigmaT$ such that, for all $\vtau\in\vSigma^+(T)$,
\begin{equation}\label{eq:uvIST}
  \uvIST\vtau\eqbydef
  \big(
  \vlproj[\cG,T]{l-1}\vtau+\vlproj[\coG,T]{m}\vtau,
  \left(\lproj[F]{k}(\vtau\SCAL\normal_{TF})\right)_{F\in\Fh[T]}
  \big).
\end{equation}
The space $\uvSigmaT$ is equipped with the $L^2(T)^d$-like norm $\norm[\vSigma,T]{{\cdot}}$ such that, for all $\uvtau\in\uvSigmaT$,
\begin{equation}\label{eq:norm.vSigmaT}
  \begin{aligned}
  \norm[\vSigma,T]{\uvtau}^2  
  &\eqbydef \norm[T]{\vtau_T}^2 + \sum_{F\in\Fh[T]} h_F\norm[F]{\tau_{TF}}^2
  \\
  &=\norm[T]{\vtau_{\cG,T}}^2 + \norm[T]{\vtau_{\coG,T}}^2 + + \sum_{F\in\Fh[T]} h_F\norm[F]{\tau_{TF}}^2,
  \end{aligned}
\end{equation}
where to pass to the second line we have used the orthogonal decomposition~\eqref{eq:decomp.PTs}.
Finally, we define the following space of local potential DOFs:
\begin{equation}\label{eq:UT}
\UT\eqbydef\Poly{l}(T).
\end{equation}

\subsection{Local reconstruction operators}

The family of mixed discretizations of problem~\eqref{eq:strong} relies on operator reconstructions defined at the element level.
Let $T\in\Th$.
The discrete divergence $\DT:\uvSigmaT\to\UT$ is such that, for all $\uvtau\in\uvSigmaT$,
\begin{equation}\label{eq:DT}
  (\DT\uvtau,q)_T
  = -(\vtau_T,\GRAD q)_T + \sum_{F\in\Fh[T]}(\tau_{TF},q)_F\qquad
  \forall q\in\UT.
\end{equation}
The right-hand side of~\eqref{eq:DT} resembles an integration by parts formula where the role of the vector function represented by $\uvtau$ in volumetric and boundary integrals is played by the element-based and face-based DOFs, respectively.

The local reconstruction $\PT:\uvSigmaT\to\cGT[k]$ of the irrotational component of the flux is such that, for all $\uvtau\in\uvSigmaT$,
\begin{equation}\label{eq:PT}
  (\PT\uvtau,\GRAD w)_T
  = -(\DT\uvtau,w)_T + \sum_{F\in\Fh[T]}(\tau_{TF},w)_F\qquad
  \forall w\in\Poly{k+1}(T),
\end{equation}
where again the right-hand side is designed to resemble an integration by parts formula where the continuous divergence operator is replaced by $\DT$, while the role of normal trace of the vector function represented by $\uvtau$ is played by boundary DOFs.

\begin{remark}\label{rem:dep.DT.FT}
  The flux DOFs $\vtau_{\coG,T}\in\coGT[m]$ do not intervene in the definitions of either $\DT$ nor $\PT$.
\end{remark}
Finally, we define the full vector field reconstruction $\ST:\uvSigmaT\to\Poly{k}(T)^d$ such that, for all $\uvtau\in\uvSigmaT$,
\begin{equation}\label{eq:ST}
  \ST\uvtau \eqbydef \PT\uvtau + \vtau_{\coG,T}.
\end{equation}
The following properties hold:
\begin{alignat}{2}\label{eq:commut.DT}
  \DT\uvIST\vtau &= \lproj[T]{l}(\DIV\vtau) &\qquad&\forall\vtau\in\vSigma^+(T),
  \\ \label{eq:cons.FT}
  \PT\uvIST\vtau &= \vtau &\qquad&\forall\vtau\in\cGT[k].
\end{alignat}
Defining the space
\begin{equation}\label{eq:cST}
  \cST
  \eqbydef\begin{cases}
  \cGT[k] & \text{if $m=0$}, \\ \Poly{k}(T)^{d} & \text{if $m=k$,}
  \end{cases}
\end{equation}
it follows from~\eqref{eq:cons.FT} together with the orthogonal decomposition~\eqref{eq:decomp.PTs} and the definitions~\eqref{eq:uvIST} of the reduction map $\uvIST$ and~\eqref{eq:ST} of $\ST$ that
\begin{equation}\label{eq:cons.ST}
  \ST\uvIST\vtau = \vtau\qquad\forall\vtau\in\cST,
\end{equation}
which expresses the polynomial consistency of $\ST$.

\subsection{Local bilinear form}

Let $T\in\Th$.
We approximate the $L^2(T)^d$-product of fluxes by means of the bilinear form $\mathrm{m}_T:\uvSigmaT\times\uvSigmaT\to\Real$ such that
\begin{subequations}
  \begin{align}\label{eq:mT}
    \mathrm{m}_T(\uvsigma,\uvtau)
    &\eqbydef (\ST\uvsigma,\ST\uvtau)_T + \mathrm{s}_{\vSigma,T}(\uvsigma,\uvtau)
    \\\label{eq:mT'}
    &= (\PT\uvsigma,\PT\uvtau)_T + (\vsigma_{\coG,T},\vtau_{\coG,T})_T + \mathrm{s}_{\vSigma,T}(\uvsigma,\uvtau),
  \end{align}
\end{subequations}
where the right-hand side is composed of a consistency and a stabilization term.

\begin{assumption}[Bilinear form $\mathrm{s}_{\vSigma,T}$]\label{ass:s0T}
  The symmetric, positive semi-definite bilinear form $\mathrm{s}_{\vSigma,T}:\uvSigmaT\times\uvSigmaT\to\Real$ satisfies the following properties:
  \begin{enumerate}[(S1)]
  \item \emph{Stability.}\label{eq:S1} It holds, for all $\uvtau\in\uvSigmaT$, with norm $\norm[\vSigma,T]{{\cdot}}$ defined by~\eqref{eq:norm.vSigmaT},
    $$
    \norm[\mathrm{m},T]{\uvtau}^2\eqbydef \mathrm{m}_T(\uvtau,\uvtau)\approx\norm[\vSigma,T]{\uvtau}^2;
    $$
  \item \emph{Polynomial consistency.}\label{eq:S2} For all $\vchi\in\cST$, with local flux reduction map $\uvIST$ defined by~\eqref{eq:uvIST},
    $$
    \mathrm{s}_{\vSigma,T}(\uvIST\vchi,\uvtau)=0\qquad\forall\uvtau\in\uvSigmaT.
    $$
  \end{enumerate}
\end{assumption}

\subsection{Global spaces and mixed problem}

We define the following global discrete spaces for the flux:
\begin{equation}\label{eq:uvSigmah}
  \uvcSigmah\eqbydef\bigtimes_{T\in\Th}\uvSigmaT,\qquad
  \uvSigmah\eqbydef\left\{
  \uvtau[h]\in\uvcSigmah\;\Big|\;
  \sum_{T\in\Th[F]}\tau_{TF}=0\quad\forall F\in\Fhi
  \right\}.
\end{equation}
The restriction of a DOF vector $\uvtau[h]\in\uvcSigmah$ to a mesh element $T\in\Th$ is denoted by $\uvtau\in\uvSigmaT$, and we equip $\uvcSigmah$ (hence also $\uvSigmah$) with the $L^2(\Omega)^d$-like norm  (cf.~\eqref{eq:norm.vSigmaT} for the definition of $\norm[\vSigma,T]{{\cdot}}$)
\begin{equation}\label{eq:norm.vSigmah}
  \norm[\vSigma,h]{\uvtau[h]}^2\eqbydef\sum_{T\in\Th}\norm[\vSigma,T]{\uvtau}^2.
\end{equation}
The global space for the potential is spanned by broken polynomials of total degree $l$:
\begin{equation}\label{eq:Uh}
  \Uh\eqbydef\Poly{l}(\Th).
\end{equation}
The global $L^2(\Omega)^d$-like product on $\uvcSigmah$ is defined by element-by-element assembly setting, for all $\uvsigma[h],\uvtau[h]\in\uvcSigmah$,
\begin{equation}\label{eq:mh}
  \mathrm{m}_h(\uvsigma[h],\uvtau[h])
  \eqbydef
  \sum_{T\in\Th}\mathrm{m}_T(\uvsigma,\uvtau).
\end{equation}
We also need the global divergence operator $\Dh:\uvcSigmah\to\Uh$ such that, for all $\uvtau[h]\in\uvcSigmah$,
$$
\restrto{(\Dh\uvtau[h])}{T} = \DT\uvtau\qquad\forall T\in\Th.
$$

\begin{problem}[Mixed problem]
  Find $(\uvsigma[h],u_h)\in\uvSigmah\times\Uh$ such that,
  \begin{subequations}\label{eq:mixed.h}
    \begin{alignat}{2}
      \label{eq:mixed.h:1}
      \mathrm{m}_h(\uvsigma[h],\uvtau[h]) + (u_h,\Dh\uvtau[h]) &= 0
      &\qquad&\forall\uvtau[h]\in\uvSigmah,    
      \\
      \label{eq:mixed.h:2}
      -(\Dh\uvsigma[h],v_h) &= (f,v_h) &\qquad&\forall v_h\in\Uh.
    \end{alignat}
  \end{subequations}
\end{problem}
Using standard arguments relying on the coercivity of $\mathrm{m}_h$ (a consequence of (S1)) and the existence of a Fortin interpolator (cf.~\eqref{eq:commut.DT}), one can prove that problem~\eqref{eq:mixed.h} is well-posed; cf., e.g.,~\cite{Boffi.Brezzi.ea:13}.

\begin{remark}[Hybridization and static condensation]\label{rem:hyb.stat.cond.mixed}
  Various possibilities are available to make the actual implementation of the method~\eqref{eq:mixed.h} more efficient.
  A first option consists in implementing the equivalent primal reformulation~\eqref{eq:hybrid.h} described in detail below; cf. also Remark~\ref{rem:stat.cond.primal}.
  Another option, in the spirit of~\cite{Araya.Harder.ea:13}, consists in locally eliminating element-based flux DOFs and element-based potential DOFs of degree $\ge 1$ by locally solving small mixed problems.
  The resulting global problem is expressed in terms of the skeletal flux DOFs plus one potential DOF per element.
\end{remark}

\subsection{Examples}\label{sec:mixed:examples}

We provide in this section a few examples of discontinuous skeletal methods originally introduced in a mixed formulation which can be traced back to~\eqref{eq:mixed.h}.
Each method is uniquely defined by prescribing the three polynomial degrees $k$, $l$, and $m$ (in accordance with~\eqref{eq:l.m}) and the expression of the local stabilization bilinear form $s_{\vSigma,T}$ for a generic mesh element $T\in\Th$.
A synopsis is provided in Table~\ref{tab:mixed}.

\begin{table}\centering
  \begin{tabular}{cccccc}
    \toprule
    Ref. & Name & $k$ & $l$ & $m$ & $\mathrm{s}_{\vSigma,T}$ \\
    \midrule
    \cite{Raviart.Thomas:77} & $\RT[0]$ Finite Element & 0 & 0 & 0 & Eq.~\eqref{eq:s0T:RT0} \\
    \cite{Brezzi.Lipnikov.ea:05} & Mimetic Finite Difference &
    \multirow{2}{*}{$0$} & \multirow{2}{*}{$0$} & \multirow{2}{*}{$0$} & \multirow{2}{*}{Eq.~\eqref{eq:s0T:MFV}} \\
    \cite{Droniou.Eymard.ea:10} & Mixed Finite Volume \\ 
    \cite{Codecasa.Specogna.ea:10} & Discrete Geometric Approach & $0$ & $0$ & $0$ & Eq.~\eqref{eq:s0T:DGA} \\
    \cite{Di-Pietro.Ern:16} & Mixed High-Order & $\ge 0$ & $k$ & $0$ & Eq.~\eqref{eq:s0T:MHO} \\    
    \cite{Brezzi.Falk.ea:14} & Mixed Virtual Element & $\ge 1$ & $k-1$ & $0$ & Eq.~\eqref{eq:s0T:VEM.mixed.BFM} \\    
    \cite{Beirao-da-Veiga.Brezzi.ea:16} & Mixed Virtual Element & $\ge 0$ & $k$ & $k$ & Eq.~\eqref{eq:s0T:VEM.mixed.BBMR} \\    
    \bottomrule
  \end{tabular}
  \caption{Examples of methods originally introduced in mixed formulation.\label{tab:mixed}}
\end{table}

\begin{example}[The Mimetic Finite Difference method of~\cite{Brezzi.Lipnikov.ea:05} and the Mixed Finite Volume method of~\cite{Droniou.Eymard.ea:10}]\label{ex:MFV}
  The Mimetic Finite Difference method of~\cite{Brezzi.Lipnikov.ea:05} and the Mixed Finite Volume method of~\cite[Section~2.3]{Droniou.Eymard.ea:10} (which is a variation of the one originally introduced in~\cite{Droniou.Eymard:06}) correspond to the choice $k=l=m=0$.
  We present them together since an equivalence result was already proved in~\cite{Droniou.Eymard.ea:10}.
  In the lowest-order case, explicit expressions can be found for both $\DT[0]$ and $\ST[0]=\PT[0]$:
  For all $\uvtau\in\uvSigmaT[0,0,0]$,
  \begin{equation}\label{eq:DT0.ST0}
    \DT[0]\uvtau = \frac{1}{\meas{T}}\sum_{F\in\Fh[T]}\meas[d-1]{F}\tau_{TF},\qquad
    \ST[0]\uvtau = \PT[0]\uvtau = \frac{1}{\meas{T}}\sum_{F\in\Fh[T]}\meas[d-1]{F}\tau_{TF}(\vec{x}_F - \vec{x}_T),
  \end{equation}
  where $\vec{x}_F$ is the barycenter of $F$ and $\vec{x}_T$ is an arbitrary point associated with $T$ which may or may not belong to $T$.
  The stabilization is parametrized by a symmetric, positive definite matrix $\matr{B}^T=(B_{FF^\prime}^T)_{F,F^\prime\in\Fh[T]}$:
    \begin{equation}\label{eq:s0T:MFV}
      \mathrm{s}_{\vSigma,T}(\uvsigma,\uvtau) =
      \sum_{F\in\Fh[T]}\sum_{F^\prime\in\Fh[T]}
      (\ST[0]\uvsigma\SCAL\normal_{TF}-\sigma_{TF}) B_{FF^\prime}^T(\ST[0]\uvtau\SCAL\normal_{TF^\prime}-\tau_{TF^\prime}).
    \end{equation}
    It is worth noting that the original Mixed Finite Volume method of~\cite{Droniou.Eymard:06} does not enter the present framework as the corresponding stabilization bilinear form
    $\mathrm{s}_{\vSigma,T}(\uvsigma,\uvtau) = \sum_{F\in\Fh[T]}h_T\meas[d-1]{F}\sigma_{TF}\tau_{TF}$
    violates (S2) (it is, however, weakly consistent).
\end{example}

\begin{example}[The lowest-order Raviart--Thomas element]\label{ex:RT}
  We assume that $T$ is an element from a matching simplicial mesh $\Th$, and consider the lowest order Raviart--Thomas space $\RT[0](T)\eqbydef\Poly{0}(T)^d+\vec{x}\Poly{0}(T)$ of \cite{Raviart.Thomas:77}.
  Clearly, the vector space $\uvSigmaT[0,0,0]$ contains the standard DOFs for $\RT[0](T)$ defined by the flux reduction map $\uvIST[0,0,0]$ as the average values of the normal components on each face.
  It can be checked that $\RT[0](T) = \opspan\left(\vec{\varphi}_F^T\right)_{F\in\Fh[T]}$ where, with $\vec{x}_T$ and $\vec{x}_F$ barycenters of $T$ and $F\in\Fh[T]$, respectively,
  $$
  \vec{\varphi}_F^T(\vec{x})\eqbydef
  \frac{\meas[d-1]{F}}{\meas{T}}(\vec{x}_F-\vec{x}_T)  
  + \frac{\meas[d-1]{F}}{d\meas{T}}(\vec{x}-\vec{x}_T)
  \qquad\forall\vec{x}\in T,
  $$
  and it holds $\restrto{(\vec{\varphi}_F^T\SCAL\normal_{TF})}{F}=1$ and $\restrto{(\vec{\varphi}_F^T\SCAL\normal_{TF'})}{F'}=0$ for all $F'\in\Fh[T]\setminus\{F\}$ (in $d=2$, this formula is a variation of \cite[Eq.~(4.3)]{Bahriawati.Carstensen:05}).
  Let $\vec{\mathfrak{t}}_T\in\RT[0](T)$ and $\uvtau=(\tau_{TF})_{F\in\Fh[T]}\eqbydef\uvIST[0,0,0]\vec{\mathfrak{t}}_T$, so that $\vec{\mathfrak{t}}_T = \sum_{F\in\Fh[T]}\vec{\varphi}_F^T\tau_{TF}$.
  Straightforward computations show that
  $$
  \DIV\vec{\mathfrak{t}}_T = \DT[0]\uvtau,\qquad  
  \vlproj[T]{0}\vec{\mathfrak{t}}_T = \ST[0]\uvtau = \PT[0]\uvtau,
  $$
  with explicit expressions for $\DT[0]$ and $\ST[0]=\PT[0]$ given by \eqref{eq:DT0.ST0}.
  Hence, we can rewrite the $L^2$-product of two functions $\vec{\mathfrak{s}}_T,\vec{\mathfrak{t}}_T\in\RT[0](T)$ with DOFs $\uvsigma\eqbydef\uvIST[0,0,0]\vec{\mathfrak{s}}_T$ and $\uvtau\eqbydef\uvIST[0,0,0]\vec{\mathfrak{t}}_T$ as follows:
  \begin{equation}\label{eq:mT.RT0}
    (\vec{\mathfrak{s}}_T,\vec{\mathfrak{t}}_T)_T
    = (\lproj[T]{0}\vec{\mathfrak{s}}_T,\lproj[T]{0}\vec{\mathfrak{t}}_T)_T
    + (\vec{\mathfrak{s}}_T-\lproj[T]{0}\vec{\mathfrak{s}}_T,\vec{\mathfrak{t}}_T-\lproj[T]{0}\vec{\mathfrak{t}}_T)_T    
    = (\ST[0]\uvsigma,\ST[0]\uvtau)_T + \mathrm{s}_{\vSigma,T}(\uvsigma,\uvtau),
  \end{equation}
  where, observing that $(\vec{\varphi}_F^T-\lproj[T]{0}\vec{\varphi}_F^T)(\vec{x})=\frac{\meas[d-1]{F}}{d\meas{T}}(\vec{x}-\vec{x}_T)$,
  \begin{equation}\label{eq:s0T:RT0}
    \mathrm{s}_{\vSigma,T}(\uvsigma,\uvtau)\eqbydef \sum_{F\in\Fh[T]}\sum_{F'\in\Fh[T]} B_{FF'}^T\sigma_{TF}\tau_{TF'},\qquad
    B_{FF'}^T\eqbydef\frac{\meas[d-1]{F}\meas[d-1]{F'}}{d^2\meas{T}^2}\int_T\norm[2]{\vec{x}-\vec{x}_T}^2\ud\vec{x}.
  \end{equation}
  From~\eqref{eq:mT.RT0} it is clear that $\mathrm{s}_{\vSigma,T}$ verifies both (S1) and (S2).
\end{example}

\begin{example}[The Discrete Geometric Approach of~\cite{Codecasa.Specogna.ea:10}]\label{ex:dga}
  Denote by $\vec{x}_T$ an arbitrary point in $T$, and assume that $T$ is star-shaped with respect to $T$.
  The Discrete Geometric Approach of~\cite{Codecasa.Specogna.ea:10} is a lowest-order method corresponding to $k=l=m=0$ based on the stable flux reconstruction such that, for all $\uvtau\in\uvSigmaT[0,0,0]$,
  \begin{equation}\label{eq:ST.dga}
    \ST[\rm dga]\uvtau\eqbydef\sum_{G\in\Fh[T]}\meas[d-1]{G}\tau_{TG}\vec{\varphi}_{TG},
  \end{equation}
  where, for all $G\in\Fh[T]$, the restriction of the basis function $\vec{\varphi}_{TG}$ to any pyramid $\PTF$ of apex $\vec{x}_T$ and base $F\in\Fh[T]$ satisfies, denoting by $\vec{x}_F$ the barycenter of $F$ and setting $\mathfrak{h}_{TF}\eqbydef{\rm dist}(\vec{x}_T,F)$,
  \begin{equation}\label{eq:varphi.dga}
    \restrto{\vec{\varphi}_{TG}}{\PTF}
    \eqbydef
    \frac{(\vec{x}_G-\vec{x}_T)}{\meas[d]{T}}
    + \left(
    \frac{(\vec{x}_F-\vec{x}_T)\otimes\normal_{TF}}{\meas[d]{T}\mathfrak{h}_{TF}} - \frac{\delta_{FG}}{\meas[d-1]{G}\mathfrak{h}_{TG}}{\Id}
    \right)(\vec{x}_T-\vec{x}_G),
  \end{equation}
  where $\delta_{FG}=1$ if $F=G$, $0$ otherwise.
  The local bilinear form $\mathrm{m}_T$ is then defined setting, for all $\uvsigma,\uvtau\in\uvSigmaT[0,0,0]$,
  \begin{equation}\label{eq:mT.dga}
    \mathrm{m}_T(\uvsigma,\uvtau)\eqbydef(\ST[\rm dga]\uvsigma,\ST[\rm dga]\uvtau)_T.
  \end{equation}
  Plugging~\eqref{eq:varphi.dga} into~\eqref{eq:ST.dga}, and using the second formula in~\eqref{eq:DT0.ST0}, we can identify in the expression of $\ST[\rm dga]$ two $L^2(T)^d$-orthogonal contributions observing that, for all $\uvtau\in\uvSigmaT[0,0,0]$ and all $F\in\Fh[T]$, it holds
  $$
  \restrto{(\ST[\rm dga]\uvtau)}{\PTF}
  = \ST[0]\uvtau + \mathfrak{h}_{TF}^{-1}(\ST[0]\vtau\SCAL\normal_{TF} - \tau_{TF})(\vec{x}_T-\vec{x}_F),
  $$
  where the first term in the right-hand side represents the consistent part of the flux, while the second acts as a stabilization.
  Hence, a straightforward computation shows that the bilinear form $\mathrm{m}_T$ defined by~\eqref{eq:mT.dga} can be recast in the form~\eqref{eq:mT} with stabilization bilinear form
  \begin{equation}\label{eq:s0T:DGA}
    \mathrm{s}_{\vSigma,T}(\uvsigma,\uvtau) =
    \sum_{F\in\Fh[T]}\frac{\norm[2]{\vec{x}_T-\vec{x}_F}^2}{d\mathfrak{h}_{TF}}(\ST[0]\uvsigma\SCAL\normal_{TF}-\sigma_{TF},\ST[0]\uvtau\SCAL\normal_{TF}-\tau_{TF})_F.
  \end{equation}
  Note that this expression can be recovered from~\eqref{eq:s0T:MFV} taking $\matr{B}^T={\rm diag}\left(\frac{\norm[2]{\vec{x}_T-\vec{x}_F}^2\meas[d-1]{F}}{d \mathfrak{h}_{TF}}\right)_{F\in\Fh[T]}$.
\end{example}

\begin{example}[The Mixed High-Order method of~\cite{Di-Pietro.Ern:16}]
  The Mixed High-Order method of~\cite{Di-Pietro.Ern:16} corresponds to the choice $l=k$ and $m=0$, for which $\ST=\PT$ holds.
  The stabilization term is defined by penalizing face-based residuals in a least-square fashion:
  \begin{equation}\label{eq:s0T:MHO}
    \mathrm{s}_{\vSigma,T}(\uvsigma,\uvtau) =
    \sum_{F\in\Fh[T]}h_F(\ST\uvsigma\SCAL\normal_{TF}-\sigma_{TF},\ST\uvtau\SCAL\normal_{TF}-\tau_{TF})_F.
  \end{equation}
  When $k=0$, this stabilization bilinear form coincides with~\eqref{eq:s0T:MFV} with $\matr{B}^T={\rm diag}(h_F\meas[d-1]{F})_{F\in\Fh[T]}$.
\end{example}

\begin{example}[The Virtual Element method of~\cite{Brezzi.Falk.ea:14}]
  Let $d=2$.
  We consider the Mixed Virtual Element method of~\cite{Brezzi.Falk.ea:14} when the diffusion tensor (denoted by $\mathbb{K}$ in the reference) is the $2\times 2$ identity matrix $\Id[2]$.
  In this case, while the DOFs for the flux~\cite[Eq. (3.8)]{Brezzi.Falk.ea:14} do not coincide with the ones in~\eqref{eq:uvSigmaT}, the resulting method~\cite[Eq.~(6.1)]{Brezzi.Falk.ea:14} can be recast in the form~\eqref{eq:mixed.h} (note, however, that this is no longer true for more general diffusion tensors).
  For a given integer $k\ge 1$, the underlying finite-dimensional local virtual space is
  \begin{multline*}
  \vfS{{\rm vem},1}(T)\eqbydef\{
  \vec{\mathfrak{t}}_T\in\Hdiv[T]\cap\Hrot[T]\st \\
  \text{%
    $\DIV\vec{\mathfrak{t}}_T\in\Poly{k-1}(T)$, $\ROT\vec{\mathfrak{t}}_T\in\Poly{k-1}(T)$, and $\restrto{\vec{\mathfrak{t}}_T}{F}\SCAL\normal_{TF}\in\Poly{k}(F)$ for all $F\in\Fh[T]$%
  }
  \},
  \end{multline*}
  where $\ROT\vec{\mathfrak{t}}_T\eqbydef\partial_1\mathfrak{t}_{T,2}-\partial_2\mathfrak{t}_{T,1}$.
  Observing that, when $\mathbb{K}=\Id[2]$, for all $\vec{\mathfrak{t}}_T\in\vfS{{\rm vem},1}(T)$, $\ROT\vec{\mathfrak{t}}_T$ does not contribute to defining $\DIV\vec{\mathfrak{t}}_T$ nor the projection on $\cGT[k]$ defined by~\cite[Eq.~(5.5)]{Brezzi.Falk.ea:14}, it can be showed that the stabilization term in~\cite[Eq.~(5.6)]{Brezzi.Falk.ea:14} actually enforces a zero-rot condition on the discrete solution.
  Hence, we can equivalently reformulate the method~\cite[Eq.~(6.1)]{Brezzi.Falk.ea:14} in terms of the zero-rot subspace
  $$
  \vfS{{\rm vem},1}(\ROT_0;T)\eqbydef\{
  \vec{\mathfrak{t}}_T\in\vfS{{\rm vem},1}(T)\st \ROT\vec{\mathfrak{t}}_T=0
  \}.
  $$
  This equivalent reformulation corresponds to the mixed form~\eqref{eq:mixed.h} with polynomial degrees $l=k-1$, and $m=0$, and stabilization bilinear form $\mathrm{s}_{\vSigma,T}$ defined as described hereafter.
  We preliminarily observe that the reduction map $\uvIST[k,k-1,0]$ (cf.~\eqref{eq:uvIST}) defines an isomorphism from $\vfS{{\rm vem},1}(\ROT_0;T)$ to $\uvSigmaT[k,k-1,0]$.
  Assume that a basis for $\uvSigmaT[k,k-1,0]$ has been fixed (a scaled monomial basis is proposed in the original reference), and denote by $\mathrm{S}^{{\rm vem},1}_{\vSigma,T}$ the bilinear form on $\vfS{{\rm vem},1}(\ROT_0;T)\times\vfS{{\rm vem},1}(\ROT_0;T)$ represented by the identity matrix in this basis.
  The stabilization bilinear form is then given by
  \begin{equation}\label{eq:s0T:VEM.mixed.BFM}
    \mathrm{s}_{\vSigma,T}(\uvsigma,\uvtau)
    \eqbydef \mathrm{S}^{{\rm vem},1}_{\vSigma,T}(\PT\uvsigma-\vec{\mathfrak{s}}_T,\PT\uvtau-\vec{\mathfrak{t}}_T)_T,
  \end{equation}
  where $\vec{\mathfrak{s}}_T$ and $\vec{\mathfrak{t}}_T$ are the unique functions of $\vfS{{\rm vem},1}(\ROT_0;T)$ such that $\uvsigma=\uvIST[k,k-1,0]\vec{\mathfrak{s}}_T$ and $\uvtau=\uvIST[k,k-1,0]\vec{\mathfrak{t}}_T$.
  This stabilization essentially corresponds to penalising in a least-square sense the high-order differences \mbox{$\vlproj[\cG,T]{k-2}(\PT\uvtau-\vtau_{\cG,T})$} and \mbox{$(\PT\uvtau\SCAL\normal_{TF}-\tau_{TF})$}, $F\in\Fh[T]$.
\end{example}

\begin{example}[The Virtual Element method of~\cite{Beirao-da-Veiga.Brezzi.ea:16}]
  A different Virtual Element method in dimension $d=2$ was presented in~\cite{Beirao-da-Veiga.Brezzi.ea:16} in the context of more general elliptic problems featuring variable diffusion as well as advective and reactive terms.
  In the pure diffusion case (which, in the original notation from the reference, corresponds to $\kappa=\Id[2]$, $\vec{b}=\vec{0}$, and $\gamma=0$), the method corresponds to the choice $l=m=k$ with $k\ge 0$.
  The underlying virtual space is, this time,
  \begin{multline*}
  \vfS{{\rm vem},2}(T)\eqbydef\{
  \vec{\mathfrak{t}}_T\in\Hdiv[T]\cap\Hrot[T]\st \\
  \text{%
    $\DIV\vec{\mathfrak{t}}_T\in\Poly{k}(T)$, $\ROT\vec{\mathfrak{t}}_T\in\Poly{k-1}(T)$, and $\restrto{(\vec{\mathfrak{t}}_T\SCAL\normal_{TF})}{F}\in\Poly{k}(F)$ for all $F\in\Fh[T]$%
  }
  \}.
  \end{multline*}
  The local flux reduction map $\uvIST[k,k,k]$ defines an isomorphism from $\vfS{{\rm vem},2}$ to $\uvSigmaT[k,k,k]$, which contains the DOF defined by~\cite[Eqs. (16)--(18)]{Beirao-da-Veiga.Brezzi.ea:16}.
  The stabilization bilinear form is defined in a similar manner as in the previous example:
  Given a bilinear form $\mathrm{S}_{\vSigma,T}^{{\rm vem},2}$ on $\vfS{{\rm vem},2}(T)\times\vfS{{\rm vem},2}(T)$ with the same scaling as the $L^2(T)^d$-inner product of fluxes, we set
  \begin{equation}\label{eq:s0T:VEM.mixed.BBMR}
    \mathrm{s}_{\vSigma,T}(\uvsigma,\uvtau)
    \eqbydef \mathrm{S}^{{\rm vem},2}_{\vSigma,T}(\ST\uvsigma-\vec{\mathfrak{s}}_T,\ST\uvtau-\vec{\mathfrak{t}}_T)_T,
  \end{equation}
  where $\vec{\mathfrak{s}}_T$ and $\vec{\mathfrak{t}}_T$ are the unique functions of $\vfS{{\rm vem},2}(T)$ such that $\uvsigma=\uvIST[k,k,k]\vec{\mathfrak{s}}_T$ and $\uvtau=\uvIST[k,k,k]\vec{\mathfrak{t}}_T$.
  This stabilization essentially corresponds to penalising in a least-square sense the high-order differences \mbox{$\vlproj[\cG,T]{k-1}(\PT\uvtau-\vtau_{\cG,T})$} and \mbox{$(\ST\uvtau\SCAL\normal_{TF}-\tau_{TF})$}, $F\in\Fh[T]$.
  For further developments on $\Hdiv$- and $\Hcurl$-conforming Virtual Elements we refer to~\cite{Beirao-da-Veiga.Brezzi.ea:15}.
\end{example}


\section{A family of primal discontinuous skeletal methods}\label{sec:primal}

We introduce in this section a family of primal discontinuous skeletal methods and provide a few examples of members of this family.

\subsection{Local space}

Let a mesh element $T\in\Th$ and three polynomial degrees $k$, $l$, and $m$ as in~\eqref{eq:l.m} be fixed.
We define the following local space for the potential:
\begin{equation*}\label{eq:uUT}
  \uUT\eqbydef\UT\times\left(
  \bigtimes_{F\in\Fh[T]}\Poly{k}(F)
  \right),
\end{equation*}
where, recalling~\eqref{eq:UT}, $\UT=\Poly{l}(T)$.
The local potential reduction map $\uIUT:H^1(T)\to\uUT$ is such that, for all $v\in H^1(T)$,
\begin{equation}\label{eq:uIUT}
  \uIUT v\eqbydef (\lproj[T]{l}v, (\lproj[F]{k}v)_{F\in\Fh[T]}).
\end{equation}
We define on $\uUT$ the $H^1(T)$-like seminorm $\norm[U,T]{{\cdot}}$ such that, for all $\uv[T]\in\uUT$,
\begin{equation}\label{eq:normUT}
  \norm[U,T]{\uv[T]}^{2}\eqbydef
  \norm[T]{\GRAD v_{T}}^{2} + \sum_{F\in\Fh[T]} h_{F}^{-1}\norm[F]{v_{F}-v_{T}}^{2},
\end{equation}
and observe that, by virtue of a local Poincar\'e inequality, the map $\norm[U,T]{{\cdot}}$ defines a norm on quotient space
\begin{equation}\label{eq:uUTs}
  \uUTs\eqbydef\uUT/\uIUT\Poly{0}(T),
\end{equation}
where two  elements of $\uUT$ belong to the same equivalence class if their difference is the interpolate of a constant function over $T$.
Clearly, $\dim(\uUTs)=\dim(\uUT)-1$.

\subsection{Local gradient reconstruction}

Let $T\in\Th$.
The family of primal methods hinges on the local gradient reconstruction operator $\GT:\uUT\to\cST$ (cf.~\eqref{eq:cST}) defined such that, for all $\uv[T]\in\uUT$,
\begin{equation}\label{eq:GT}
  (\GT\uv[T],\vtau)_{T}
  = -(v_T,\DIV\vtau)_{T} + \sum_{F\in\Fh[T]} (v_F, \vtau\SCAL\normal_{TF})_{F}\qquad
  \forall\vtau\in\cST,
\end{equation}
where the right-hand side is devised so as to resemble an integration by parts formula where the role of the function represented by $\uv[T]$ inside volumetric and boundary terms is played by element- and face-based DOFs, respectively.
\begin{remark}[Polynomial degree $m$]
  The polynomial degree $m$ does not intervene in the definition~\eqref{eq:uUT} of the local space of potential DOFs.
  Its role is to determine the arrival space for the discrete gradient operator $\GT$ which, recalling~\eqref{eq:cST}, is either $\cGT[k]$ (if $m=0$) or $\Poly{k}(T)^d$ (if $m=k$).
\end{remark}
Adapting the arguments of~\cite[Lemma~3]{Di-Pietro.Ern.ea:14} (cf., in particular, Eq. (17) therein), it can be checked that the following commuting property holds: For all $v\in H^1(T)$,
\begin{equation}\label{eq:commut.GT}
  \GT\uIUT v = \vlproj[\cS,T]{k,m} \GRAD v,
\end{equation}
where $\vlproj[\cS,T]{k,m}$ denotes the $L^2$-orthogonal projector on $\cST$ and the potential reduction map $\uIUT$ is defined by~\eqref{eq:uIUT}.

\subsection{Local bilinear form}

We define, for all $T\in\Th$, the local bilinear form $\mathrm{a}_T:\uUT\times\uUT\to\Real$ as follows:
\begin{equation}\label{eq:primal.aT}
  \mathrm{a}_T(\uu[T],\uv[T])\eqbydef (\GT\uu[T],\GT\uv[T])_{T} + \mathrm{s}_{U,T}(\uu[T],\uv[T]),
\end{equation}
where, as for the bilinear form $\mathrm{m}_T$ defined by~\eqref{eq:mT}, the right-hand side is composed of a consistency and a stabilization term.
\begin{assumption}[Bilinear form $\mathrm{s}_{U,T}$]\label{ass:s1T}
  The symmetric, positive semi-definite bilinear form $\mathrm{s}_{U,T}:\uUT\times\uUT\to\Real$ satisfies the following properties:  
  \begin{enumerate}[(S1\sprime)]
  \item \emph{Stability.} It holds, for all $\uv[T]\in\uUT$, with seminorm $\norm[U,T]{{\cdot}}$ defined by~\eqref{eq:normUT},
    $$
    \norm[\mathrm{a},T]{\uv[T]}^{2}\eqbydef \mathrm{a}_T(\uv[T],\uv[T])\approx\norm[U,T]{\uv[T]}^{2}.
    $$
  \item \emph{Polynomial consistency.} For all $w\in\Poly{k+1}(T)$, with local potential reduction map $\uIUT$ defined by~\eqref{eq:uIUT},
    $$
    \mathrm{s}_{U,T}(\uIUT w,\uv[T])=0\qquad\forall\uv[T]\in\uUT.
    $$
  \end{enumerate}
\end{assumption}

\subsection{Global space and primal problem}

We define the following global spaces of potential DOFs with single-valued interface unknowns:
\begin{equation}\label{eq:uUh}
  \uUh\eqbydef\Uh\times\left(
  \bigtimes_{F\in\Fh}\Poly{k}(F)
  \right),\qquad
  \uUhD\eqbydef\left\{
  \uv\in\uUh\st v_F=0\quad\forall F\in\Fhb
  \right\},
\end{equation}
where the subspace $\uUhD$ embeds the homogeneous Dirichlet boundary condition.
For a generic DOF vector $\uv\in\uUh$ we use the notation $\uv=( (v_T)_{T\in\Th}, (v_F)_{F\in\Fh} )$, and we denote by $\uv[T]\in\uUT$ its restriction to $T$.
We also denote by $v_h\in\Poly{l}(\Th)$ the piecewise polynomial function such that $\restrto{v_h}{T}=v_T$ for all $T\in\Th$.
On $\uUh$, we define the global $H^1(\Omega)$-like seminorm $\norm[U,h]{{\cdot}}$ such that, for all $\uv\in\uUh$,
\begin{equation}\label{eq:norm.U}
  \norm[U,h]{\uv}^{2}\eqbydef\sum_{T\in\Th}\norm[U,T]{\uv[T]}^{2},
\end{equation}
with $\norm[U,T]{{\cdot}}$ given by~\eqref{eq:normUT}.
Following a reasoning analogous to that of~\cite[Proposition~5]{Di-Pietro.Ern:15*1}, it can be easily checked that the map $\norm[U,h]{{\cdot}}$ defines a norm on $\uUhD$.
We will also need the global potential reduction map $\uIUh:H^{1}(\Omega)\to\uUh$ such that, for all $v\in H^{1}(\Omega)$,
\begin{equation*}\label{eq:uIUh}
  \uIUh v = ( (\lproj[T]{l} v)_{T\in\Th}, (\lproj[F]{k} v)_{F\in\Fh}).
\end{equation*}
Clearly, the restriction of $\uIUh$ to a mesh element $T\in\Th$ coincides with the local potential reduction map defined by~\eqref{eq:uIUT}.
Also, $\uIUh$ maps elements of $H_0^1(\Omega)$ to elements of $\uUhD$.
Finally, we define the global bilinear form $\mathrm{a}_h:\uUh\times\uUh\to\Real$ by element-by-element assembly setting
\begin{equation*}\label{eq:primal.ah}
  \mathrm{a}_h(\uu,\uv)\eqbydef\sum_{T\in\Th}\mathrm{a}_T(\uu[T],\uv[T]).
\end{equation*}

\begin{problem}[Primal problem]
Find $\uu\in\uUhD$ such that
\begin{equation}\label{eq:primal.h}
  \mathrm{a}_h(\uu,\uv) = (f,v_h)\qquad\forall\uv[h]\in\uUhD.
\end{equation}
\end{problem}

\begin{remark}[Static condensation]\label{rem:stat.cond.primal}
  In the actual implementation of the method~\eqref{eq:primal.h}, element-based DOFs can be locally eliminated by static condensation.
  The procedure is essentially analogous to the one described, e.g., in~\cite[Section~2.4]{Cockburn.Di-Pietro.ea:15}, to which we refer for further details.
\end{remark}

\subsection{Examples}
  
We collect in this section a few examples of discontinuous skeletal methods originally introduced in a primal formulation which can be traced back to~\eqref{eq:primal.h}.
Each method is uniquely defined by prescribing the three polynomial degrees $k$, $l$, and $m$ (in accordance with~\eqref{eq:l.m}) and the expression of the local stabilization bilinear form $s_{U,T}$ for a generic mesh element $T\in\Th$.
A synopsis is provided in Table~\ref{tab:primal}.

\begin{table}\centering
  \begin{tabular}{cccccc}
    \toprule
    Ref. & Name & $k$ & $l$ & $m$ & $\mathrm{s}_{U,T}$ \\
    \midrule
    \cite{Eymard.Gallouet.ea:10} & Hybrid Finite Volume & $0$ & $0$ & $0$ & Eq.~\eqref{eq:s1T:HFV} \\
    \cite{Droniou.Eymard.ea:10} & Hybrid Finite Volume & $0$ & $0$ & $0$ & Eq.~\eqref{eq:s1T:HFV'} \\    
    \cite{Lehrenfeld:10} & Hybridizable Discontinuous Galerkin & $\ge 0$ & $k+1$ & $k$ & Eq.~\eqref{eq:s1T:HDG} \\
    \cite{Cockburn.Di-Pietro.ea:15} & Hybridizable Discontinuous Galerkin & $\ge 0$ & Eq.~\eqref{eq:l.m} & $k$ & Eq.~\eqref{eq:s1T:HHO} \\        
    \cite{Di-Pietro.Ern.ea:14} & Hybrid High-Order & $\ge 0$ & $k$ & $0$ & Eq.~\eqref{eq:s1T:HHO} \\
    \cite{Cockburn.Di-Pietro.ea:15} & Hybrid High-Order & $\ge 0$ & Eq.~\eqref{eq:l.m} & $0$ & Eq.~\eqref{eq:s1T:HHO} \\    
    \cite{Lipnikov.Manzini:14,Ayuso-de-Dios.Lipnikov.ea:15} & High-Order Mimetic & $\ge 0$\textsuperscript{*} & $k-1$ & $0$ & Eq.~\eqref{eq:s1T:HOM} \\
    \bottomrule
  \end{tabular}
  \caption{Examples of methods originally introduced in primal formulation. \textsuperscript{*}~The High-Order Mimetic method enters the present framework only for $k\ge 1$.\label{tab:primal}}
\end{table}

\begin{example}[The Hybrid Finite Volume method of~\cite{Eymard.Gallouet.ea:10} and its generalization of~\cite{Droniou.Eymard.ea:10}]\label{eq:HFV}
  The Hybrid Finite Volume method of~\cite[Section~2.1]{Eymard.Gallouet.ea:10} corresponds to $k=l=m=0$.
  In this case, an explicit expression for the gradient operator $\GT[0]$ defined by~\eqref{eq:GT} is available:
  For all $\uv[T]\in\uUT[0,0]$,
  \begin{equation}\label{eq:GT0}
    \GT[0]\uv[T]=\frac{1}{\meas{T}}\sum_{F\in\Fh[T]}\meas[d-1]{F}v_F\normal_{TF}.
  \end{equation}
  For every element $T\in\Th$, the stabilization bilinear form is such that
  \begin{equation}\label{eq:s1T:HFV}
    \mathrm{s}_{U,T}(\uu[T],\uv[T])
    = \sum_{F\in\Fh[T]}\meas[d-1]{F}\frac{\eta}{\mathfrak{h}_{TF}}\deTF[0]\uu[T]\deTF[0]\uv[T],
  \end{equation}
  where $\eta>0$ is a user-dependent stabilization parameter, $\mathfrak{h}_{TF}$ as in Example~\ref{ex:dga} and the face-based residual operator $\deTF[0]:\uUT[0,0]\to\Poly{0}(F)$ is such that, denoting by $\vec{x}_F$ the barycenter of $F$ and by $\vec{x}_T$ an arbitrary point associated with $T$ which may or may not belong to $T$,
  \begin{equation}\label{eq:deltaTF0}
    \deTF[0]\uv[T] \eqbydef v_T + \GT[0]\uv[T]\SCAL(\vec{x}_F-\vec{x}_T) - v_F.
  \end{equation}
  In~\cite[Section~2.2]{Droniou.Eymard.ea:10}, the following generalization of~\eqref{eq:s1T:HFV} is proposed:
  For a given positive definite matrix $\matr{B}^T=(B_{FF^\prime}^T)_{F,F^\prime\in\Fh[T]}$,
  \begin{equation}\label{eq:s1T:HFV'}
    \mathrm{s}_{U,T}(\uu[T],\uv[T])
    = \sum_{F\in\Fh[T]}\sum_{F^\prime\in\Fh[T]} \deTF[0]\uu[T] B_{FF^\prime}^T\delta_{TF^\prime}^0\uv[T].
  \end{equation}
\end{example}

\begin{example}[The Hybrid High-Order method of~\cite{Di-Pietro.Ern.ea:14} and the variants of~\cite{Cockburn.Di-Pietro.ea:15}]\label{ex:HHO}
  The original Hybrid High-Order method of~\cite{Di-Pietro.Ern.ea:14} corresponds to the choice $l=k$ and $m=0$. 
  In~\cite{Cockburn.Di-Pietro.ea:15}, variants corresponding to $l=k-1$ (when $k\ge 1$) and $l=k+1$ have also been proposed.
  Let an element $T\in\Th$ be fixed, and define the potential reconstruction operator $\pT:\uUT\to\Poly{k+1}(T)$ such that, for all $\uv[T]\in\uUT$,
  \begin{equation}\label{eq:pT}
    \text{$\GRAD\pT\uv[T]=\GT\uv[T]$\quad and\quad $(\pT\uv[T]-v_T,1)_T=0$.}
  \end{equation}
  Note that the first condition makes sense since, having supposed $m=0$, $\GT\uv[T]\in\cGT$.
  The stabilization bilinear form is defined as follows:
  \begin{equation}\label{eq:s1T:HHO}
    \mathrm{s}_{U,T}(\uu[T],\uv[T]) = \sum_{F\in\Fh[T]} h_F^{-1}(\delta_{TF}^{k}\uu[T],\delta_{TF}^{k}\uv[T])_F,
  \end{equation}
  where, for all $F\in\Fh[T]$, the face-based residual operator $\deTF:\uUT\to\Poly{k}(F)$ is such that, for all $\uv[T]\in\uUT$,
  \begin{equation}\label{eq:deltaTFkl}
    \deTF\uv[T]=\lproj[F]{k}\left(\pT\uv[T] - v_F - \lproj[T]{l}(\pT\uv[T] - v_T)\right).
  \end{equation}
  As already observed in~\cite[Section~2.5]{Di-Pietro.Ern.ea:14}, in the lowest-order case $k=0$ the face-based residuals defined by~\eqref{eq:deltaTF0} and~\eqref{eq:deltaTFkl} coincide, and the stabilization~\eqref{eq:s1T:HHO} can be recovered from~\eqref{eq:s1T:HFV'} selecting $\matr{B}^T={\rm diag}(h_F^{-1}\meas[d-1]{F})_{F\in\Fh[T]}$ (the only difference with respect to \eqref{eq:s1T:HFV} is the change of local scaling $\mathfrak{h}_{TF}\gets h_F$).
\end{example}

\begin{example}[The Crouzeix--Raviart finite element]\label{ex:NC}
  Let $T$ be an element belonging to a matching simplicial mesh $\Th$ with barycenter $\vec{x}_T$, and consider the Crouzeix--Raviart element $\NC(T)$ of~\cite{Crouzeix.Raviart:73}.
  We study the solution of problem~\eqref{eq:primal.h} using the Hybrid Finite Volume method of Example~\ref{eq:HFV} (or, equivalently, the Hybrid High-Order method of Example~\ref{ex:HHO} with $k=l=m=0$) but with right-hand side discretized as
  \begin{equation}\label{eq:rhs.CR}
    \sum_{T\in\Th}(f,\pT[1]\uv[T])_T,
  \end{equation}
  where the potential reconstruction $\pT[1]$ is defined according to~\eqref{eq:pT} but with average value on $T$ set to $\frac1{d+1}\sum_{F\in\Fh[T]} v_F$ (here, $\mathfrak{h}_{TF}$ is the orthogonal distance of $\vec{x}_T$ from $F$).
  We start by noticing that it holds $\lproj[F]{0}\pT[1]\uv[T] = \pT[1]\uv[T](\vec{x}_F) = v_F$ for all $\uv[T]\in\uUT$ and all $F\in\Fh[T]$ with $\vec{x}_F$ barycenter of $F$.
  As a consequence, for the face-based residual operator~\eqref{eq:deltaTF0} it holds for all $\uv[T]\in\uUT[0,0]$ that
  $$
  \deTF[0]\uv[T] = -\lproj[T]{0}(\pT[1]\uv[T] - v_T) = v_T - \pT[1]\uv[T](\vec{x}_T).
  $$
  Then, observing that element-based DOFs do not contribute to the consistency term in~\eqref{eq:primal.aT} nor to the right-hand side, we infer that the stabilization term is actually enforcing the condition $\pT[1]\uv[T](\vec{x}_T) = v_T$ for all $T\in\Th$.
  As a result, denoting by $\uu[h]\in\uUhD[0,0]$ the solution of problem~\eqref{eq:primal.h} with right-hand side modified as in~\eqref{eq:rhs.CR}, the piecewise affine field equal to $\pT[1]\uu[T]$ inside each mesh element $T\in\Th$ coincides with the Crouzeix--Raviart solution~\eqref{eq:nonc}.
\end{example}

\begin{example}[The Hybridizable Discontinuous Galerkin method of~\cite{Lehrenfeld:10} and the variants of~\cite{Cockburn.Di-Pietro.ea:15}]
  The Hybridizable Discontinuous Galerkin originally proposed in~\cite[Remark~1.2.4]{Lehrenfeld:10} corresponds to the case $l=k+1$ and $m=k$ and stabilization
  \begin{equation}\label{eq:s1T:HDG}
    \mathrm{s}_{U,T}(\uu[T],\uv[T])
    = \sum_{F\in\Fh[T]}h_F^{-1}(\lproj[F]{k}(u_T-u_F), \lproj[F]{k}(v_T-v_F))_F.
  \end{equation}
  As pointed out in~\cite[Remark~2]{Cockburn.Di-Pietro.ea:15}, this stabilization coincides with~\eqref{eq:s1T:HHO} when $l=k+1$.
  Motivated by this remark, variants corresponding to the choices $l=k-1$ (when $k\ge 1$) and $l=k$ and $m=k$ are proposed therein.
  It is worth noting here that the original Hybridizable Discontinuous Galerkin method of~\cite{Castillo.Cockburn.ea:00,Cockburn.Gopalakrishnan.ea:09} does not fit in the present framework since the corresponding stabilization bilinear form is only polynomially consistent up to degree $k$, i.e., it does not satisfy (S2\sprime).
  Correspondingly, the  orders of convergence are reduced (cf. \cite[Table 1]{Cockburn.Di-Pietro.ea:15} for further details).
\end{example}

\begin{example}[The High-Order Mimetic method of~\cite{Lipnikov.Manzini:14,Ayuso-de-Dios.Lipnikov.ea:15}]
  The High-Order Mimetic method of~\cite{Lipnikov.Manzini:14} (subsequently referred to as Nonconforming Virtual Element method in~\cite{Ayuso-de-Dios.Lipnikov.ea:15}) provides a high-order generalization of the concepts underlying Mimetic Difference Methods (cf., e.g.,~\cite{Beirao-da-Veiga.Lipnikov.ea:14}).
  Its lowest-order version, corresponding to the case $k=0$ and $l=-1$, violates~\eqref{eq:l.m}, and therefore does not enter our unified framework.
  For $k\ge 1$, on the other hand, it corresponds to the choices $l=k-1$ and $m=0$.
  To write the corresponding bilinear form, define the finite-dimensional local virtual space
  $$
  \mathfrak{U}^k(T)\eqbydef\left\{
  \text{%
    $\mathfrak{v}_T\in H^1(T)\st\LAPL\mathfrak{v}_T\in\Poly{k-1}(T)$ and $\restrto{(\GRAD\mathfrak{v}_T)}{F}\SCAL\normal_{TF}\in\Poly{k}(F)$ for all $F\in\Fh[T]$%
  }
  \right\}.
  $$
  Clearly, $\Poly{k+1}(T)\subset\mathfrak{U}^k(T)$, and it can be proved that $\uIUT[k,k-1]$ defines an isomorphism from $\mathfrak{U}^k(T)$ to $\uUT[k,k-1]$.
  Denote by $\mathrm{S}_T^{\rm hom}:\mathfrak{U}^k(T)\times\mathfrak{U}^k(T)\to\Real$ a bilinear form whose representation in the canonical basis of $\mathfrak{U}^k(T)$ is spectrally equivalent to the unit matrix.
  The stabilization bilinear form is obtained setting, for all $\uu[T],\uv[T]\in\uUT$,
  \begin{equation}\label{eq:s1T:HOM}
    \mathrm{s}_{U,T}(\uu[T],\uv[T])
    \eqbydef h^{d-2}_T\mathrm{S}_T^{\rm hom}(\pT\uu[T] - \mathfrak{u}_T, \pT\uv[T] - \mathfrak{v}_T),
  \end{equation}
  where $\mathfrak{u}_T$ and $\mathfrak{v}_T$ are the unique functions in $\mathfrak{U}^k(T)$ such that $\uu[T]=\uIUT[k,k-1]\mathfrak{u}_T$ and $\uv[T]=\uIUT[k,k-1]\mathfrak{v}_T$, while the operator $\pT$ is defined by~\eqref{eq:pT}.
  The stabilization~\eqref{eq:s1T:HOM} essentially corresponds to penalizing in a least-square sense the high-order differences \mbox{$\lproj[T]{l}(\pT\uv[T] - v_T)$} and \mbox{$\lproj[F]{k}(\pT\uv[T] - v_F)$}, $F\in\Fh[T]$, with scaling factor choosen so that the uniform equivalence in (S1\sprime) holds.
\end{example}


\section{From mixed to primal methods}\label{sec:m->p}

In this section we obtain from~\eqref{eq:mixed.h} an equivalent primal problem by hybridization.
The primal hybrid problem is then shown to belong to the family~\eqref{eq:primal.h} of primal discontinuous skeletal methods.

\subsection{Mixed hybrid formulation of mixed methods}

We define the bilinear form $\mathrm{b}_h:\uvcSigmah\times\uUh\to\Real$ (with spaces $\uvcSigmah$ and $\uUh$ defined by~\eqref{eq:uvSigmah} and~\eqref{eq:uUh}, respectively) such that, for all $(\uvtau[h],\uv)\in\uvcSigmah\times\uUh$,
\begin{equation}\label{eq:bT.bh}
  \mathrm{b}_h(\uvtau[h],\uv)\eqbydef\sum_{T\in\Th}\mathrm{b}_T(\uvtau[T],\uv[T]),\qquad
  \mathrm{b}_T(\uvtau,\uv[T])\eqbydef (\DT\uvtau,v_T)_T - \sum_{F\in\Fh[T]}(\tau_{TF},v_F)_F.
\end{equation}
For further use, we note that it holds for all $T\in\Th$, all $\uvtau\in\uvSigmaT$, and all $\uv[T]\in\uUT$,
\begin{equation}\label{eq:bT'}
  \mathrm{b}_T(\uvtau,\uv[T])
  = -(\vtau_{\cG,T},\GRAD v_T)_T + \sum_{F\in\Fh[T]}(\tau_{TF}, v_T - v_F)_F,
\end{equation}
as can be easily checked replacing $\DT$ by its definition~\eqref{eq:DT} and accounting for Remark~\ref{rem:dep.DT.FT}.
Hence, using the Cauchy--Schwarz inequality and recalling the definitions~\eqref{eq:norm.vSigmaT} and~\eqref{eq:normUT} of $\norm[\vSigma,T]{{\cdot}}$ and $\norm[U,T]{{\cdot}}$, we infer the following boundedness result for $\mathrm{b}_T$:
\begin{equation}\label{eq:bT.cont}
  |\mathrm{b}_T(\uvtau,\uv[T])|\le\norm[\vSigma,T]{\uvtau}\norm[U,T]{\uv[T]}.
\end{equation}
\begin{problem}[Mixed hybrid problem]
  Find $(\uvsigma[h],\uu)\in\uvcSigmah\times\uUhD$ such that, 
  \begin{subequations}\label{eq:mixed.hyb.h}
    \begin{alignat}{2}
      \label{eq:mixed.hyb.h:1}
      \forall T\in\Th,\qquad
      \mathrm{m}_T(\uvsigma,\uvtau) + \mathrm{b}_T(\uvtau,\uu[T]) &= 0
      &\qquad&\forall\uvtau\in\uvSigmaT,
      \\
      \label{eq:mixed.hyb.h:2}
      -\mathrm{b}_h(\uvsigma[h],\uv) &= (f,v_h)
      &\qquad&\forall\uv\in\uUhD.
    \end{alignat}
  \end{subequations}
\end{problem}
Compared to the mixed problem~\eqref{eq:mixed.h}, the single-valuedness of interface flux unknowns is enforced here by Lagrange multipliers (corresponding to the skeletal DOFs in $\uUhD$) instead of being embedded in the discrete space.
Equation~\eqref{eq:mixed.hyb.h:1} defines a set of local constitutive relations connecting flux to potential DOFs inside each mesh element.
Equation~\eqref{eq:mixed.hyb.h:2}, on the other hand, expresses local balances and a global transmission condition.
In what follows, we will eliminate flux unknowns by locally inverting~\eqref{eq:mixed.hyb.h:1}, ending up with a problem in the hybrid potential unknowns only.

\subsection{Mixed-to-primal potential-to-flux operator}

For all $T\in\Th$, we define the local mixed-to-primal potential-to-flux operator $\vsT:\uUT\to\uvSigmaT$ such that, for all $\uv[T]\in\uUT$,
\begin{equation}\label{eq:vsT}
  \mathrm{m}_T(\vsT\uv[T],\uvtau) = -\mathrm{b}_T(\uvtau,\uv[T])\qquad
  \forall\uvtau\in\uvSigmaT.
\end{equation}
Recalling the reformulation~\eqref{eq:bT'} of $\mathrm{b}_T$,~\eqref{eq:vsT} equivalently rewrites
\begin{equation}\label{eq:vsT'}
  \mathrm{m}_T(\vsT\uv[T],\uvtau)
  = (\GRAD v_T,\vtau_{\cG,T})_T + \sum_{F\in\Fh[T]}(v_F-v_T,\tau_{TF})_F\qquad
  \forall\uvtau\in\uvSigmaT.
\end{equation}
We next state some useful properties for the potential-to-flux operator.
\begin{lemma}[Properties of the mixed-to-primal potential-to-flux operator]\label{lem:vsT}
  Let a mesh element $T\in\Th$ be given and let $\mathrm{s}_{\vSigma,T}$ be a bilinear form satisfying Assumption~\ref{ass:s0T}.
  Then, the corresponding potential-to-flux operator $\vsT$ given by~\eqref{eq:vsT} is well defined and has the following properties:
  \begin{enumerate}[1)]
  \item \emph{Stability and continuity.} For all $\uv[T]\in\uUT$, it holds 
    \begin{equation}\label{eq:vsT.cont}
      \norm[\vSigma,T]{\vsT\uv[T]}\approx\norm[U,T]{\uv[T]},
    \end{equation}
    with norms $\norm[\vSigma,T]{{\cdot}}$ and $\norm[U,T]{{\cdot}}$ defined by~\eqref{eq:norm.vSigmaT} and~\eqref{eq:normUT}, respectively.
  \item \emph{Commuting property.} For all $w\in\Poly{k+1}(T)$, we have
    \begin{equation}\label{eq:vsT.commuting}
      \vsT\uIUT w = \uvIST\GRAD w.
    \end{equation}
  \item \emph{Link with the discrete gradient operator.} It holds, with operators $\GT$ and $\ST$ defined by~\eqref{eq:GT} and~\eqref{eq:ST}, respectively, that
    \begin{equation}\label{eq:char.GT}
      \GT\eqbydef\ST\circ\vsT.
    \end{equation}  
  \end{enumerate}
\end{lemma}
\begin{proof}
  Problem~\eqref{eq:vsT} is well-posed owing to assumption (S1) expressing the coercivity of $\mathrm{m}_T$. As a result, $\vsT$ is well defined.
    \begin{asparaenum}[1)]
    \item \emph{Stability and continuity.} Using (S1) followed by the definition~\eqref{eq:vsT} of $\vsT$ and the boundedness~\eqref{eq:bT.cont} of $\mathrm{b}_T$, we infer, for all $\uv[T]\in\uUT$,
      \begin{equation}\label{eq:vsT.cont:1}
        \norm[\vSigma,T]{\vsT\uv[T]}^2
        \lesssim\norm[\mathrm{m},T]{\vsT\uv[T]}^2
        = -\mathrm{b}_T(\vsT\uv[T],\uv[T])
        \le\norm[\vSigma,T]{\vsT\uv[T]}\norm[U,T]{\uv[T]}.
      \end{equation}
      To prove the converse inequality, let $\uvtau\in\uvSigmaT$ in~\eqref{eq:vsT'} be such that $\vtau_T=\GRAD v_T$ and $\tau_{TF}=h_F^{-1}(v_F-v_T)$ for all $F\in\Fh[T]$, and observe that
      \begin{equation}\label{eq:vsT.cont:2}
        \norm[U,T]{\uv[T]}^2
        =\mathrm{m}_T(\vsT\uv[T],\uvtau)
        \lesssim\norm[\vSigma,T]{\vsT\uv[T]}\norm[\vSigma,T]{\uvtau}
        =\norm[\vSigma,T]{\vsT\uv[T]}\norm[U,T]{\uv[T]},
      \end{equation}
      where we have used the Cauchy--Schwarz inequality together with (S1) to bound $\mathrm{m}_T$ and the definitions~\eqref{eq:norm.vSigmaT} of $\norm[\vSigma,T]{{\cdot}}$ and~\eqref{eq:norm.U} of $\norm[U,T]{{\cdot}}$ to infer $\norm[\vSigma,T]{\uvtau}=\norm[U,T]{\uv[T]}$ and conclude.    
    \item \emph{Commuting property.} Let $w\in\Poly{k+1}(T)$.
      Using the definition~\eqref{eq:vsT} of $\vsT$ with $\uv[T]=\uIUT w$ and recalling~\eqref{eq:bT.bh}, we infer, for all $\uvtau\in\uvSigmaT$,
      \begin{equation}\label{eq:char.vsT:1}
        \begin{aligned}
          \mathrm{m}_T(\vsT\uIUT w,\uvtau)
          &=
          -(\lproj[T]{l} w, \DT\uvtau)_T + \sum_{F\in\Fh[T]}(\lproj[F]{k}w, \tau_{TF})_F
          \\
          &=
          -(w, \DT\uvtau)_T + \sum_{F\in\Fh[T]}(w, \tau_{TF})_F
          = (\GRAD w, \PT\uvtau)_T,
        \end{aligned}
      \end{equation}
      where we have used the definitions~\eqref{eq:lproj} of $\lproj[T]{l}$ and $\lproj[F]{k}$ to pass to the second line and the definition~\eqref{eq:PT} of $\PT$ to conclude.
      On the other hand, using the definition~\eqref{eq:mT} of $\mathrm{m}_T$ followed by the polynomial consistency~\eqref{eq:cons.ST} of $\ST$ together with (S2), for all $\uvtau\in\uvSigmaT$ we have that
      \begin{equation}\label{eq:char.vsT:2}
        \begin{aligned}
          \mathrm{m}_T(\uvIST\GRAD w,\uvtau)
          &= (\ST\uvIST\GRAD w, \ST\uvtau)_T + \mathrm{s}_{\vSigma,T}(\uvIST\GRAD w,\uvtau)
          \\
          &= (\GRAD w, \ST\uvtau)_T
          = (\GRAD w, \PT\uvtau)_T,
        \end{aligned}
      \end{equation}
      where the last equality follows from the definition~\eqref{eq:ST} of $\ST$ together with the orthogonal decomposition~\eqref{eq:decomp.PTs}.
      Subtracting~\eqref{eq:char.vsT:2} from~\eqref{eq:char.vsT:1}, we infer, for all $\uvtau\in\uvSigmaT$,
      $$
      \mathrm{m}_T(\vsT\uIUT w-\uvIST\GRAD w,\uvtau) = 0,
      $$
      from which~\eqref{eq:vsT.commuting} follows since $\mathrm{m}_T$ is coercive on $\uvSigmaT$ owing to (S1).\qedhere

    \item \emph{Link with the discrete gradient operator.}
      Let $\uv[T]\in\uUT$, $\vtau\in\cST$, and set $\uvtau\eqbydef\uvIST\vtau$.
      Recalling the definition~\eqref{eq:mT} of $\mathrm{m}_T$, and using the polynomial consistency~\eqref{eq:cons.ST} of $\ST$ together with (S2), it is readily inferred that
      \begin{equation}\label{eq:char.GT:1}
        \mathrm{m}_T(\vsT\uv[T],\uvtau) = ((\ST\circ\vsT)\uv[T],\vtau)_{T}.
      \end{equation}
      On the other hand, recalling the definitions~\eqref{eq:uvIST} of $\uvIST$ and~\eqref{eq:bT.bh} of $\mathrm{b}_T$, we get
      \begin{equation}\label{eq:char.GT:2}
        \begin{aligned}
          \mathrm{b}_T(\uvtau,\uv[T])
          &= (v_{T}, \DT\uvtau)_{T} - \sum_{F\in\Fh[T]} (v_{F}, \tau_{TF})_{F}
          \\
          &= (v_{T}, \lproj[T]{l}(\DIV\vtau))_{T} - \sum_{F\in\Fh[T]}(v_{F}, \lproj[F]{k}(\vtau\SCAL\normal_{TF}))_{F}
          \\
          &= (v_{T}, \DIV\vtau)_{T} - \sum_{F\in\Fh[T]}(v_{F}, \vtau\SCAL\normal_{TF})_{F}
          = -(\GT\uv,\vtau)_{T},
        \end{aligned}
      \end{equation}
      where we have used the commuting property~\eqref{eq:commut.DT} of $\DT$ in the second line and the definition~\eqref{eq:lproj} of $\lproj[T]{l}$ and $\lproj[F]{k}$ and~\eqref{eq:GT} of $\GT$ in the third.
      To conclude, plug~\eqref{eq:char.GT:1} and~\eqref{eq:char.GT:2} into the definition~\eqref{eq:vsT} of $\vsT$.\qedhere
    \end{asparaenum}
\end{proof}
\subsection{Equivalent primal formulations of mixed methods}

We start by showing a link among problems~\eqref{eq:mixed.h},~\eqref{eq:mixed.hyb.h}, and the following
\begin{problem}[Primal hybrid problem]
  Find $(\uvsigma[h],\uu)\in\uvcSigmah\times\uUhD$ such that
  \begin{subequations}\label{eq:hybrid.h}
    \begin{equation}\label{eq:hybrid.h:1}
      \uvsigma[T]=\vsT\uu[T]\qquad\forall T\in\Th,
    \end{equation}
    with potential-to-flux operator $\vsT$ defined by~\eqref{eq:vsT} and $\uu[h]$ solution of
    \begin{equation}\label{eq:hybrid.h:2}
      \mathrm{a}_h(\uu,\uv) = (f,v_h) \qquad \forall\uv\in\uUhD,
    \end{equation}
  \end{subequations}
  where the bilinear form $\mathrm{a}_h$ on $\uUh\times\uUh$ is such that
  \begin{equation}\label{eq:hybrid.ah}
    \mathrm{a}_h(\uu,\uv)
    \eqbydef\sum_{T\in\Th} \mathrm{a}_T(\uu[T],\uv[T]),\qquad
    \mathrm{a}_T(\uu[T],\uv[T])\eqbydef\mathrm{m}_T(\vsT\uu[T],\vsT\uv[T]).
  \end{equation}
\end{problem}
The well-posedness of~\eqref{eq:hybrid.h:2} is an immediate consequence of point 1) in Theorem~\ref{thm:equivalence:m->p} below.
\begin{theorem}[Link among the mixed, mixed hybrid and primal hybrid problems]\label{thm:hybridization}
  For all $T\in\Th$, let $\mathrm{s}_{\vSigma,T}$ satisfy Assumption~\ref{ass:s0T}.
  Let $(\uvsigma[h],\uu)\in\uvcSigmah\times\uUhD$, and let $u_h\in\Uh$ be such that $\restrto{u_h}{T}=u_T$ for all $T\in\Th$.
  Then, the following statements are equivalent:
    \begin{enumerate}[(i)]
    \item $(\uvsigma[h],\uu)$ solves the mixed hybrid problem~\eqref{eq:mixed.hyb.h};
    \item $\uvsigma[h]\in\uvSigmah$ and $(\uvsigma[h],u_h)$ solves the mixed problem~\eqref{eq:mixed.h};
    \item $(\uvsigma[h],\uu)$ solves the primal hybrid problem~\eqref{eq:hybrid.h}.
    \end{enumerate}
\end{theorem}
\begin{proof}
  The equivalence $\text{(i)}\iff\text{(ii)}$ classically follows from the theory of Lagrange multipliers.
  Let us prove the equivalence $\text{(i)}\iff\text{(iii)}$.
  We first show that if $(\uvsigma[h],\uu[h])$ solves the mixed hybrid problem~\eqref{eq:mixed.hyb.h}, then it solves the primal hybrid problem~\eqref{eq:hybrid.h}.
  Equation~\eqref{eq:hybrid.h:1} immediately follows from~\eqref{eq:mixed.hyb.h:1} recalling the definition~\eqref{eq:vsT} of the potential-to-flux operator.
  As a consequence, it holds for all $T\in\Th$ and all $\uv[T]\in\uUT$,
  $$
  -\mathrm{b}_T(\uvsigma[T],\uv[T])
  = -\mathrm{b}_T(\vsT\uu[T],\uv[T])
  = \mathrm{m}_T(\vsT\uu[T],\vsT\uv[T])
  = \mathrm{a}_T(\uu[T],\uv[T]),
  $$
  where we have used the definition~\eqref{eq:vsT} of the potential-to-flux operator together with the symmetry of $\mathrm{m}_T$ in the second equality and the definition~\eqref{eq:hybrid.ah} of the bilinear form $\mathrm{a}_T$ to conclude.
  This implies that~\eqref{eq:mixed.hyb.h:2} is equivalent to~\eqref{eq:hybrid.h:2}.
  By similar arguments, we can prove that if $(\uvsigma[h],\uu[h])$ solves the primal hybrid problem~\eqref{eq:hybrid.h}, then it solves the mixed hybrid problem~\eqref{eq:mixed.hyb.h}, thus concluding the proof.
\end{proof}
We close this section with our main result, viz. the existence of a primal method belonging to the family~\eqref{eq:primal.h} whose solution coincides with that of the mixed method~\eqref{eq:mixed.h} for given stabilization bilinear forms satisfying Assumption~\ref{ass:s0T}.
In the light of Theorem~\ref{thm:hybridization}, it suffices to state the equivalence with respect to the corresponding mixed hybrid formulation~\eqref{eq:mixed.hyb.h}.
\begin{theorem}[Link with the family of primal discontinuous skeletal methods]\label{thm:equivalence:m->p}
  For all $T\in\Th$, let $\mathrm{s}_{\vSigma,T}$ satisfy Assumption~\ref{ass:s0T} and set with $\vsT$ defined by~\eqref{eq:vsT}:
  \begin{equation}\label{eq:sUT:m->p}
    \mathrm{s}_{U,T}(\uu[T],\uv[T])\eqbydef\mathrm{s}_{\vSigma,T}(\vsT\uu[T],\vsT\uv[T]).
  \end{equation}
  Then,
  \begin{enumerate}[1)]
  \item \emph{Properties of $\mathrm{s}_{U,T}$.}
    The stabilization bilinear forms $\mathrm{s}_{U,T}$, $T\in\Th$, satisfy Assumption~\ref{ass:s1T};
  \item \emph{Link with primal methods.} $\uu[h]\in\uUhD$ solves the primal problem~\eqref{eq:primal.h} with stabilization as in~\eqref{eq:sUT:m->p} if and only if $(\uvsigma[h],\uu[h])\in\uvcSigmah\times\uUhD$ with $\uvsigma[h]$ such that $\uvsigma[T]=\vsT\uu[T]$ for all $T\in\Th$ solves the mixed hybrid problem~\eqref{eq:mixed.hyb.h}.
  \end{enumerate}
\end{theorem}
\begin{proof}%
  \begin{asparaenum}[1)]
  \item \emph{Properties of $\mathrm{s}_{U,T}$.}
    Let $T\in\Th$.
    The bilinear form $\mathrm{s}_{U,T}$ is clearly symmetric and positive semi-definite. It then suffices to prove conditions (S1\sprime) and (S2\sprime).
    To prove (S1\sprime), observe that for all $\uv[T]\in\uUT$ we have
  $$
  \norm[\mathrm{a},T]{\uv[T]}
  = \norm[\mathrm{m},T]{\vsT\uv[T]}
  \approx\norm[\vSigma,T]{\vsT\uv[T]}
  \approx\norm[U,T]{\uv[T]},
  $$
  where we have used the definition~\eqref{eq:hybrid.ah} of $\mathrm{a}_T$, (S1), and the stability and continuity~\eqref{eq:vsT.cont} of $\vsT$.
  Let us prove (S2\sprime). Letting $w\in\Poly{k+1}(T)$, for all $\uv[T]\in\uUT$ we have
    $$
    \mathrm{s}_{U,T}(\uIUT w,\uv[T])
    = \mathrm{s}_{\vSigma,T}(\vsT\uIUT w,\vsT\uv[T])
    = \mathrm{s}_{\vSigma,T}(\uvIST\GRAD w,\vsT\uv[T])
    = 0,
    $$
    where we have used the definition~\eqref{eq:sUT:m->p} of $\mathrm{s}_{U,T}$, the commuting property~\eqref{eq:vsT.commuting}, and (S2).

  \item \emph{Link with primal methods.} Compare the primal hybrid formulation~\eqref{eq:hybrid.h} with the primal formulation~\eqref{eq:primal.h} and recall the equivalence with the mixed hybrid formulation~\eqref{eq:mixed.hyb.h} stated in Theorem~\ref{thm:hybridization}.\qedhere
    \end{asparaenum}
\end{proof}


\section{From primal to mixed methods}\label{sec:p->m}

In this section we show that the primal discontinuous skeletal methods of Section~\ref{sec:primal} with $m=0$ can be recast into the mixed formulation introduced in Section~\ref{sec:mixed}.
This enables us to close the circle and show a precise equivalence relation between the family~\eqref{eq:mixed.h} of mixed discontinuous skeletal methods and the family~\eqref{eq:primal.h} of primal discontinuous skeletal methods.

\subsection{Primal-to-mixed potential-to-flux operator}

We assume from this point on that, for a given integer $k\ge 0$, $l$ is as in~\eqref{eq:l.m} and $$m=0.$$
The crucial ingredient is the primal-to-mixed potential-to-flux operator $\vsT[k,l]:\uUT\to\uvSigmaT[k,l,0]$ such that, for all $\uw[T]\in\uUT$, $\vsT[k,l]\uw[T]$ solves
\begin{equation}\label{eq:primal.vsT}
  -\mathrm{b}_T(\vsT[k,l]\uw[T],\uv[T]) = \mathrm{a}_T(\uw[T],\uv[T])\qquad\forall\uv[T]\in\uUT.
\end{equation}
The use of a similar notation as for the mixed-to-primal potential-to-flux operator is motivated by the fact that these two operators share the same properties (compare Lemmas~\ref{lem:vsT} and~\ref{lem:primal.vsT}) and play very much the same role.
\begin{lemma}[Properties of the primal-to-mixed potential-to-flux operator]\label{lem:primal.vsT}
  Let a mesh element $T\in\Th$ be given and let $\mathrm{s}_{U,T}$ be a bilinear form satisfying Assumption~\ref{ass:s1T}.
  Then, the corresponding potential-to-flux operator $\vsT[k,l]$ given by~\eqref{eq:vsT} is well defined and has the following properties:
  \begin{enumerate}[1)]
  \item \emph{Stability and continuity.} For all $\uv[T]\in\uUT$, it holds with norms $\norm[\vSigma,T]{{\cdot}}$ and $\norm[U,T]{{\cdot}}$ defined by~\eqref{eq:norm.vSigmaT} and~\eqref{eq:normUT}, respectively,
    \begin{equation}\label{eq:primal.vsT.cont}
      \norm[\vSigma,T]{\vsT[k,l]\uv[T]}\approx\norm[U,T]{\uv[T]}.
    \end{equation}
  \item \emph{Commuting property.} For all $w\in\Poly{k+1}(T)$, we have
    \begin{equation}\label{eq:primal.vsT.commuting}
      \vsT[k,l]\uIUT w = \uvIST[k,l,0]\GRAD w.
    \end{equation}
  \item \emph{Link with the discrete gradient operator.} It holds, with operators $\GT$, $\PT$, and $\ST$ defined by~\eqref{eq:GT},~\eqref{eq:PT}, and~\eqref{eq:ST}, respectively, that
    \begin{equation}\label{eq:primal.char.GT}
      \GT=\PT\circ\vsT[k,l]=\ST\circ\vsT[k,l].
    \end{equation}  
  \end{enumerate}
  Additionally, $\vsT[k,l]$ defines an isomorphism from $\uUTs$ (cf.~\eqref{eq:uUTs}) to $\uvSigmaT[k,l,0]$.   
\end{lemma}
\begin{proof}
  Let $T\in\Th$.
  To show that $\vsT[k,l]$ is well defined we prove the following inf-sup condition: For all $\uvtau\in\uvSigmaT[k,l,0]$,
  \begin{equation}\label{eq:inf-sup}
    \norm[\vSigma,T]{\uvtau}\le
    \mathsf{S}\eqbydef\sup_{\uv[T]\in\uUTs\setminus\{\underline{0}_{U,T}\}}\frac{\mathrm{b}_T(\uvtau,\uv[T])}{\norm[U,T]{\uv[T]}}.
  \end{equation}
  Let $\uv[\vtau,T]\in\uUT$ be such that $\GRAD v_{\vtau,T} = \vtau_T$ and $v_{\vtau,F} - v_{\vtau,T}=h_F\tau_{TF}$ ($\uv[\vtau,T]$ is defined up to an element of $\uIUT\Poly{0}(T)$, coeherently with the fact that we write $\uUTs$ in the supremum).
  It can be checked that $\norm[U,T]{\uv[\vtau,T]}=\norm[\vSigma,T]{\uvtau[T]}$ and it holds, recalling the reformulation~\eqref{eq:bT'} of the bilinear form $\mathrm{b}_T$,
  $$
  \norm[\vSigma,T]{\uvtau}^2
  = -\mathrm{b}_T(\uvtau[T],\uv[\vtau,T])
  \le\mathsf{S}\norm[U,T]{\uv[\vtau,T]}
  =\mathsf{S}\norm[\vSigma,T]{\uvtau},
  $$
  which proves~\eqref{eq:inf-sup}.
  To prove the well-posedness of problem~\eqref{eq:primal.vsT} it only remains to observe that, for all $\uv[T]\in\uIUT\Poly{0}(T)$, equation~\eqref{eq:primal.vsT} becomes the trivial identity $0=0$, which can be intepreted as a compatibility condition.
  Finally, the fact that $\vsT[k,l]$ defines an isomorphism from $\uUTs$ to $\uvSigmaT[k,l,0]$ follows observing that $\vsT[k,l]$ is injective as a result of~\eqref{eq:inf-sup} and $\dim(\uUTs)=\dim(\uvSigmaT[k,l,0])$.
  
  \begin{asparaenum}[1)]
  \item \emph{Stability and continuity.} Combining the inf-sup condition~\eqref{eq:inf-sup} with the definition~\eqref{eq:primal.vsT} of $\vsT[k,l]$, and using the Cauchy--Schwarz inequality followed by (S1), we get for all $\uv[T]\in\uUT$ that
    $$
    \norm[\vSigma,T]{\vsT[k,l]\uv[T]}
    \le\sup_{\uw[T]\in\uUTs\setminus\{\underline{0}_{U,T}\}}\frac{\mathrm{b}_T(\vsT[k,l]\uv[T],\uw[T])}{\norm[U,T]{\uw[T]}}
    =\sup_{\uw[T]\in\uUTs\setminus\{\underline{0}_{U,T}\}}\frac{\mathrm{a}_T(\uv[T],\uw[T])}{\norm[U,T]{\uw[T]}}
    \lesssim\norm[U,T]{\uv[T]}.
    $$
    On the other hand, (S1) followed by the definition~\eqref{eq:primal.vsT} of $\vsT[k,l]$ and the boundedness~\eqref{eq:bT.cont} of the bilinear form $\mathrm{b}_T$ yields
    $$
    \norm[U,T]{\uv[T]}^2
    \lesssim\mathrm{a}_T(\uv[T],\uv[T])
    = -\mathrm{b}_T(\vsT[k,l]\uv[T],\uv[T])
    \le\norm[\vSigma,T]{\vsT[k,l]\uv[T]}\norm[U,T]{\uv[T]},
    $$
    which concludes the proof of~\eqref{eq:primal.vsT.cont}.
    
  \item \emph{Commuting property.} Let $w\in\Poly{k+1}(T)$.
    For all $\uv[T]\in\uUT$ it holds
    $$
    \begin{aligned}
      -\mathrm{b}_T(\vsT[k,l]\uIUT w,\uv[T])
      = \mathrm{a}_T(\uIUT w,\uv[T])
      = (\GRAD w,\GT\uv[T])_T
      = -\mathrm{b}_T(\uvIST[k,l,0]\GRAD w,\uv[T]),
    \end{aligned}
    $$
    where we have used the definition~\eqref{eq:primal.vsT} of $\vsT[k,l]$ in the first equality, the definition~\eqref{eq:primal.aT} of $\mathrm{a}_T$ together with (S2\sprime) in the second equality, and concluded recalling the definitions~\eqref{eq:GT} of $\GT$,~\eqref{eq:uvIST} of $\uvIST[k,l,0]$, and~\eqref{eq:bT.bh} of $\mathrm{b}_T$.
    As a consequence,
    $$
    \mathrm{b}_T(\uvIST[k,l,0]\GRAD w-\vsT[k,l]\uIUT w,\uv[T]) = 0\qquad\forall \uv[T]\in\uUT,
    $$
    which, accounting for the inf-sup condition~\eqref{eq:inf-sup}, implies~\eqref{eq:primal.char.GT}.
    
  \item \emph{Link with the discrete gradient operator.}
    Let $\uv[T]\in\uUT$ and $w\in\Poly{k+1}(T)$.
    Recalling the definitions~\eqref{eq:bT.bh} of $\mathrm{b}_T$ and~\eqref{eq:uIUT} of $\uIUT$, we infer that
    $$
    \begin{aligned}
      -\mathrm{b}_T(\vsT[k,l]\uv[T],\uIUT w)
      &=-(\DT\vsT[k,l]\uv[T],\lproj[T]{l} w)_T + \sum_{F\in\Fh[T]}((\vsT[k,l]\uv[T])_{TF},\lproj[F]{k} w)_F
      \\
      &=-(\DT\vsT[k,l]\uv[T], w)_T + \sum_{F\in\Fh[T]}((\vsT[k,l]\uv[T])_{TF}, w)_F
      = ((\PT\circ\vsT[k,l])\uv[T],\GRAD w)_T,
    \end{aligned}
    $$
    where we have used the definition~\eqref{eq:lproj} of $\lproj[T]{l}$ and $\lproj[F]{k}$ to pass to the second line and the definition~\eqref{eq:PT} of $\PT$ to conclude.
    On the other hand, by the definition~\eqref{eq:primal.aT} of $\mathrm{a}_T$ together with the polynomial consistency of $\GT$ (a consequence of~\eqref{eq:commut.GT}) and (S2\sprime), we have
    $$
    \mathrm{a}_T(\uv[T],\uIUT w) = (\GT\uv[T],\GRAD w)_T.
    $$
    Substituting the above relations into the definition~\eqref{eq:primal.vsT} of $\vsT[k,l]$ we infer that $\GT\uv[T]=\PT\circ\vsT[k,l]$.
    Additionally, since we have supposed $m=0$, we also have $\ST=\PT$, thus concluding the proof.
    \qedhere
  \end{asparaenum}
\end{proof}

\subsection{Equivalent mixed formulation of primal methods}

We close this section by showing the existence of a mixed method belonging to the family~\eqref{eq:mixed.h} whose solution coincides with that of the primal problem~\eqref{eq:primal.h}.
In the light of Theorem~\ref{thm:hybridization}, we state the equivalence result in terms of the corresponding mixed hybrid formulation~\eqref{eq:mixed.hyb.h}.

\begin{theorem}[Link with the family of mixed discontinuous skeletal methods]\label{thm:equivalence:p->m}
  For all $T\in\Th$, let $\mathrm{s}_{U,T}$ satisfy Assumption~\ref{ass:s1T} and set, for all $\uvsigma[T],\uvtau[T]\in\uvSigmaT[k,l,0]$,
  \begin{equation}\label{eq:s0T:p->m}
    \mathrm{s}_{\vSigma,T}(\uvsigma[T],\uvtau[T])
    \eqbydef\mathrm{s}_{U,T}((\vsT[k,l])^{-1}\uvsigma[T],(\vsT[k,l])^{-1}\uvtau[T]),
  \end{equation}
  where it is understood that $(\vsT[k,l])^{-1}\uvtau[T]$ and $(\vsT[k,l])^{-1}\uvsigma[T]$ are defined up to an element of $\uIUT\Poly{0}(T)$.
  Then,
  \begin{enumerate}[1)]
  \item \emph{Properties of $\mathrm{s}_{\vSigma,T}$.} The stabilization bilinear forms $\mathrm{s}_{\vSigma,T}$, $T\in\Th$ satisfy Assumption~\ref{ass:s0T};
  \item \emph{Link with mixed methods.} $(\uvsigma[h],\uu[h])\in\uvcSigmah[k,l,0]\times\uUhD$ solves the mixed hybrid problem~\eqref{eq:mixed.hyb.h} with stabilization as in~\eqref{eq:s0T:p->m} if and only if $\uu[h]$ solves the primal problem~\eqref{eq:primal.h} and, for all $T\in\Th$, $\uvsigma[T]=\vsT[k,l]\uu[T]$ with $\vsT[k,l]$ defined by~\eqref{eq:primal.vsT}.
  \end{enumerate}
\end{theorem}
\begin{proof}
  \begin{asparaenum}[1)]
  \item \emph{Properties of $\mathrm{s}_{\vSigma,T}$.}
    Let $T\in\Th$.
    The bilinear form $\mathrm{s}_{\vSigma,T}$ is clearly symmetric and positive semi-definite. It then suffices to prove conditions (S1) and (S2).
    Let us start by (S1).
    Recalling the definition~\eqref{eq:mT} of the bilinear form $\mathrm{m}_T$, property~\eqref{eq:primal.char.GT} for the potential-to-flux operator $\vsT[k,l]$ defined by~\eqref{eq:primal.vsT}, and~\eqref{eq:s0T:p->m}, we infer for all $\uw[T],\uv[T]\in\uUT$ that
    \begin{equation}\label{eq:primal.mT.aT}
      \mathrm{m}_T(\vsT[k,l]\uw[T],\vsT[k,l]\uv[T]) = \mathrm{a}_T(\uw[T],\uv[T]).
    \end{equation}
    Let now $\uvtau[T]\in\uvSigmaT[k,l,0]$ be such that $\uvtau[T]=\vsT[k,l]\uv[T]$ with $\uv[T]\in\uUT$ (the existence of such $\uv[T]$, defined up to an element of $\uIUT\Poly{0}(T)$, follows from Lemma~\ref{lem:primal.vsT}).
    We have that
    $$
      \norm[\vSigma,T]{\uvtau[T]}
      \approx\norm[U,T]{\uv[T]}
      \approx\norm[\mathrm{a},T]{\uv[T]} = \norm[\mathrm{m},T]{\uvtau[T]},
    $$
    where the first norm equivalence follows from~\eqref{eq:primal.vsT.cont}, the second from (S2\sprime), and the last one from~\eqref{eq:primal.mT.aT}.
    Property (S1) follows.

    Let us now prove (S2).
    Let $\vchi\in\cGT[k]$ be such that $\vchi=\GRAD w$ with $w\in\Poly{k+1}(T)$.
    For all $\uv[T]\in\uUT$ it holds,
    $$
    \mathrm{s}_{\vSigma,T}(\uvIST[k,l,0]\vchi,\vsT[k,l]\uv[T])
    = \mathrm{s}_{U,T}((\vsT[k,l])^{-1}\uvIST[k,l,0]\vchi,\uv[T])
    = \mathrm{s}_{U,T}(\uIUT w,\uv[T])
    =0,
    $$   
    where we have used the definition~\eqref{eq:s0T:p->m} of $\mathrm{s}_{\vSigma,T}$, the commuting property~\eqref{eq:primal.vsT.commuting}, and concluded using (S2\sprime).

  \item \emph{Link with mixed methods.}
    We let $(\uvsigma[h],\uu[h])\in\uvcSigmah[k,l,0]\times\uUhD$ solve the mixed hybrid problem~\eqref{eq:mixed.hyb.h} with $s_{\vSigma,T}$ given by~\eqref{eq:s0T:p->m}, and we show that $\uu[h]$ solves \eqref{eq:primal.h} and $\uvsigma[T]=\vsT[k,l]\uu[T]$ for all $T\in\Th$.
    Making $\uvtau[T]=\vsT[k,l]\uv[T]$ with $\uv[T]\in\uUT$ in~\eqref{eq:mixed.hyb.h:1}, it is inferred
    $$
    \begin{aligned}
      0
      = \mathrm{m}_T(\uvsigma[T],\vsT[k,l]\uv[T]) + \mathrm{b}_T(\vsT[k,l]\uv[T],\uu[T])
      = \mathrm{m}_T(\uvsigma[T]-\vsT[k,l]\uu[T],\vsT[k,l]\uv[T]).
    \end{aligned}
    $$
    Since $\uvSigmaT[k,l,0]=\vsT[k,l]\uUT$ as a result of Lemma~\ref{lem:primal.vsT} and $\uv[T]$ is arbitrary in $\uUT$, this means that
    \begin{equation}\label{eq:equivalence:p->:4}
      \uvsigma[T]=\vsT[k,l]\uu[T]\qquad\forall T\in\Th.
    \end{equation}
    Plugging this relation into~\eqref{eq:mixed.hyb.h:2}, and recalling the definition~\eqref{eq:primal.vsT} of $\vsT[k,l]$, we infer that it holds for all $\uv[h]\in\uUhD$,
    $$
    (f,v_h)
    = - \sum_{T\in\Th}\mathrm{b}_T(\uvsigma[T],\uv[T])
    = - \sum_{T\in\Th}\mathrm{b}_T(\vsT[k,l]\uu[T],\uv[T])
    = \mathrm{a}_h(\uu[h],\uv[h]),
    $$
    which shows that $\uu[h]$ solves the primal problem~\eqref{eq:primal.h}.
    Following a similar reasoning one can prove that, if $\uu[h]$ solves~\eqref{eq:primal.h}, then $(\uvsigma[h],\uu[h])$ with $\uvsigma[T]=\vsT[k,l]\uu[T]$ for all $T\in\Th$ solves~\eqref{eq:mixed.hyb.h}.
    \qedhere
  \end{asparaenum}
\end{proof}


\section{Analysis}\label{sec:analysis}

In this section we carry out a unified convergence analysis encompassing both mixed and primal discontinuous skeletal methods.
Recalling Theorems~\ref{thm:hybridization},~\ref{thm:equivalence:m->p}, and~\ref{thm:equivalence:p->m}, we focus on the mixed hybrid problem~\eqref{eq:mixed.hyb.h}.
Let three integers $k\ge 0$ and $l,m$ as in~\eqref{eq:l.m} be fixed, set $\uvXh\eqbydef\uvcSigmah\times\uUhD$, and define the bilinear form $\mathcal{A}_h:\uvXh\times\uvXh\to\Real$ such that
\begin{equation}\label{eq:cAh}
  \mathcal{A}_h((\uvsigma[h],\uu[h]),(\uvtau[h],\uv[h]))
  \eqbydef
  \mathrm{m}_h(\uvsigma[h],\uvtau[h]) + \mathrm{b}_h(\uvtau[h],\uu[h]) - \mathrm{b}_h(\uvsigma[h],\uv[h]).
\end{equation}
Problem~\eqref{eq:mixed.hyb.h} admits the following equivalent reformulation:
Find $(\uvsigma[h],\uu)\in\uvcSigmah\times\uUhD$ such that,
\begin{equation}\label{eq:mixed.hyb.h'}
  \mathcal{A}_h((\uvsigma[h],\uu[h]),(\uvtau[h],\uv[h])) = (f,v_h)\qquad\forall (\uvtau[h],\uv[h])\in\uvcSigmah\times\uUhD.
\end{equation}
\subsection{Stability and well-posedness}
We equip the space $\uvXh$ with the norm $\norm[\vX,h]{{\cdot}}$ such that, for all \mbox{$(\uvtau[h],\uv[h])\in\uvXh$},
$$
\norm[\vX,h]{(\uvtau[h],\uv[h])}^2\eqbydef
\norm[\vSigma,h]{\uvtau[h]}^2 + \norm[U,h]{\uv[h]}^2,
$$
with norms $\norm[\vSigma,h]{{\cdot}}$ on $\uvcSigmah$ and $\norm[U,h]{{\cdot}}$ on $\uUh$ defined by~\eqref{eq:norm.vSigmah} and~\eqref{eq:norm.U}, respectively.
\begin{lemma}[Well-posedness]
  For all $(\uvchi[h],\uw[h])\in\uvXh$ it holds
  \begin{equation}\label{eq:cAh:inf-sup}
    \norm[\vX,h]{(\uvchi[h],\uw[h])}
    \lesssim\sup_{(\uvtau[h],\uv[h])\in\uvXh\setminus\{\underline{\vec{0}}_{\vX,h}\}}
    \frac{\mathcal{A}_h((\uvchi[h],\uw[h]),(\uvtau[h],\uv[h]))}{\norm[\vX,h]{(\uvtau[h],\uv[h])}}.
  \end{equation}
  Consequently, problem~\eqref{eq:mixed.hyb.h'} is well-posed.
\end{lemma}
\begin{proof}
  We start by proving the following inf-sup condition for $\mathrm{b}_h$:
  For all $\uv[h]\in\uUhD$,
  \begin{equation}\label{eq:bh:inf-sup}
    \norm[U,h]{\uv[h]}\lesssim
    \sup_{\uvtau[h]\in\uvcSigmah\setminus\{\underline{\vec{0}}_{\vSigma,h}\}}\frac{\mathrm{b}_h(\uvtau[h],\uv[h])}{\norm[\vSigma,h]{\uvtau[h]}}.
  \end{equation}
  Fix an element $\uv[h]\in\uUhD$, and let $\uvtau[v,h]\in\uvcSigmah$ be such that, for all $T\in\Th$, $\vtau_{v,T}=\GRAD v_T$ and $\tau_{v,TF}=h_F^{-1}(v_F - v_T)$.
  Denoting by $\mathsf{S}$ the supremum in~\eqref{eq:bh:inf-sup} from~\eqref{eq:bT'} it is inferred that
  $$
  \norm[U,h]{\uv[h]}^2
  =\mathrm{b}_h(\uvtau[v,h],\uv[h])
  \le\mathsf{S}\norm[\vSigma,h]{\uvtau[v,h]},
  $$
  and~\eqref{eq:bh:inf-sup} readily follows observing that, by the definitions~\eqref{eq:norm.vSigmaT} and~\eqref{eq:normUT} of the local norms, $\norm[\vSigma,T]{\uvtau[v,T]}=\norm[U,T]{\uv[T]}$.
  The inf-sup condition~\eqref{eq:cAh:inf-sup} on $\mathcal{A}_h$ and the well-posedness of problem~\eqref{eq:mixed.hyb.h} are then classical consequences of the $\norm[\vSigma,h]{{\cdot}}$-coercivity of $\mathrm{m}_h$ (itself a consequence of (S1)) and the inf-sup condition~\eqref{eq:bh:inf-sup} on $\mathrm{b}_h$; cf., e.g.,~\cite{Boffi.Brezzi.ea:13}.
\end{proof}
\subsection{Energy error estimate}

We estimate the error defined as the difference between the solution of the mixed hybrid problem~\eqref{eq:mixed.hyb.h} and the projection $(\uvhsigma[h],\uuh)\in\uvcSigmah\times\uUhD$ of the exact solution defined as follows:
$$
\uvhsigma[h]\eqbydef\uvISh\GRADh\cu[h]\quad\forall T\in\Th,\qquad
\uuh\eqbydef\uIUh u,
$$
where $\cu[h]\in\Poly{k+1}(\Th)$ is such that, for all $T\in\Th$, $\cu[T]\eqbydef\restrto{\cu[h]}{T}$ is the local elliptic projection of $u$ satisfying
\begin{equation}\label{eq:ell.proj}
  \text{$\GRAD\cu[T]=\vlproj[\cG,T]{k}\GRAD u$\quad and\quad $(\cu[T]-u,1)_T=0$,}
\end{equation}
while $\uvISh$ is the global flux reduction map on $\uvcSigmah$ whose restriction to every mesh elements $T\in\Th$ coincides with $\uvIST$ defined by~\eqref{eq:uvIST}.
Optimal approximation properties for $\cu[h]$ on admissible mesh sequence are proved in~\cite[Lemma~3]{Di-Pietro.Ern.ea:14} and, in a more general framework, in~\cite{Di-Pietro.Droniou:16}.
\begin{theorem}[Energy error estimate]\label{thm:en.err.est}
  Let $u\in H^1_0(\Omega)$ be the weak solution of problem~\eqref{eq:strong}, and assume the additional regularity $u\in H^{k+2}(\Omega)$.
  Then, it holds
  \begin{equation}\label{eq:en.err.est}
    \norm[\vX,h]{(\uvsigma[h]-\uvhsigma[h],\uu[h]-\uuh)}
    \lesssim h^{k+1}\norm[H^{k+2}(\Omega)]{u}.
  \end{equation}
\end{theorem}
\begin{proof}
  The following error equation descends from~\eqref{eq:mixed.hyb.h'}:
  For all $(\uvtau[h],\uv[h])\in\uvcSigmah\times\uUhD$,
  \begin{equation*}\label{eq:err.eq}
    \mathcal{A}_h((\uvsigma[h]-\uvhsigma[h],\uu[h]-\uuh),(\uvtau[h],\uv[h]))
    = \mathcal{E}_h(\uvtau[h],\uv[h]),
  \end{equation*}
  with consistency error
  \begin{equation}\label{eq:Eh}
    \mathcal{E}_h(\uvtau[h],\uv[h])
    \eqbydef (f, v_h) + \mathrm{b}_h(\uvhsigma[h],\uv[h]) - \mathrm{m}_h(\uvhsigma[h],\uvtau[h]) - \mathrm{b}_h(\uvtau[h],\uuh).
  \end{equation}
  Recalling the inf-sup condition~\eqref{eq:cAh:inf-sup}, we then have that
  \begin{equation}\label{eq:basic.en.est}
    \norm[\vX,h]{(\uvsigma[h]-\uvhsigma[h],\uu[h]-\uuh)}
    \lesssim\sup_{(\uvtau[h],\uv[h])\in\uvXh\setminus\{\underline{\vec{0}}_{\vX,h}\}}
    \frac{\mathcal{E}_h(\uvtau[h],\uv[h])}{\norm[\vX,h]{(\uvtau[h],\uv[h])}}.
  \end{equation}
  To conclude, it suffices to bound $\mathcal{E}_h(\uvtau[h],\uv[h])$.
  Denote by $\term_1,\ldots,\term_4$ the addends in the right-hand side of~\eqref{eq:Eh}.
  Recalling that $f=-\LAPL u$ a.e. in $\Omega$, integrating by parts element-by-element, and using the fact that the normal component of $\GRAD u$ is continuous across all interfaces $F\in\Fhi$ and that $v_F$ vanishes on boundary faces $F\in\Fhb$, we have that
  $$
  \term_1 = \sum_{T\in\Th}\left(
  (\GRAD u, \GRAD v_T) + \sum_{F\in\Fh[T]}(\GRAD u\SCAL\normal_{TF}, v_F - v_T)_F
  \right).
  $$
  Using the commuting property~\eqref{eq:commut.DT} of $\DT$ to infer $\DT\uvhsigma=\LAPL\cu[T]$, and integrating by parts element-by-element, we have that
  $$
  \term_2 = -\sum_{T\in\Th}\left(
  (\GRAD u,\GRAD v_T)_T + \sum_{F\in\Fh[T]}(\GRAD\cu\SCAL\normal_{TF},v_F-v_T)_F
  \right),
  $$
  where we have used the definition~\eqref{eq:ell.proj} of $\cu$ to write $\GRAD u$ instead of $\GRAD\cu$ in the first term.
  The Cauchy--Schwarz inequality yields
  \begin{equation}\label{eq:en.est.T1.T2}
    |\term_1 + \term_2|
    \le\left(
      \sum_{F\in\Fh[T]}h_F\norm[F]{\GRAD(u-\cu)}^2
      \right)^{\frac12}\hspace{-1ex}\times\left(
      \sum_{F\in\Fh[T]}h_F^{-1}\norm[F]{v_F - v_T}^2
    \right)^{\frac12}
    \lesssim h^{k+1}\norm[H^{k+2}(\Omega)]{u}\norm[U,h]{\uv[h]},
  \end{equation}
  where we have used the optimal approximation properties of $\cu$ to conclude.
  
Recalling the definition~\eqref{eq:mT'} of $\mathrm{m}_T$, using the polynomial consistency~\eqref{eq:cons.FT} of $\PT$ together with (S2), and expanding $\PT\uvtau[T]$ according to its definition~\eqref{eq:PT} (with $w=\cu$), it is inferred that
  $$
  \term_3 = -\sum_{T\in\Th}(\GRAD\cu,\PT\uvtau)_T
  = \sum_{T\in\Th}\left(
  (\cu,\DT\uvtau)_T - \sum_{F\in\Fh[T]}(\cu,\tau_{TF})_F
  \right).
  $$
  Recalling~\eqref{eq:bT'} together with the definitions~\eqref{eq:uIUT} of $\uIUT$ and~\eqref{eq:lproj} of $\lproj[T]{l}$ and $\lproj[F]{k}$, we get that
  $$ 
  \term_4 = \sum_{T\in\Th}\left(
  -(u,\DT\uvtau)_T + \sum_{F\in\Fh[T]}(u,\tau_{TF})_F
  \right).
  $$
  Using the Cauchy--Schwarz inequality, we then obtain
  \begin{equation}\label{eq:en.est.T3.T4}
    \begin{aligned}
      |\term_3 + \term_4|
      &\le\left[
      \sum_{T\in\Th}\left(h_T^{-2}\norm[T]{u-\cu}^2 + \hspace{-0.5em}\sum_{F\in\Fh[T]}h_F^{-1}\norm[F]{u-\cu}^2\right)
      \right]^{\frac12}\hspace{-1ex}\times\left[
      \sum_{T\in\Th}\left(h_T^2\norm[T]{\DT\uvtau}^2 + \hspace{-0.5em}\sum_{F\in\Fh[T]}h_F\norm[F]{\tau_{TF}}^2\right)
      \right]^{\frac12}
      \\
      &\lesssim h^{k+1}\norm[H^{k+2}(\Omega)]{u}\norm[\vSigma,T]{\uvtau},
    \end{aligned}
  \end{equation}
  where we have used the optimal approximation properties of $\cu$ and the inverse inequality $\norm[T]{\DT\uvtau}\lesssim h_T^{-1}\norm[\vSigma,T]{\uvtau}$ to pass to the second line.
Combining~\eqref{eq:en.est.T1.T2} with~\eqref{eq:en.est.T3.T4}, we infer the bound
  $$
  |\mathcal{E}_h(\uvtau[h],\uv[h])|\lesssim h^{k+1}\norm[H^{k+2}(\Omega)]{u}\norm[\vX,h]{(\uvtau[h],\uv[h])},
  $$
  which, plugged into~\eqref{eq:basic.en.est}, yields the desired result.
\end{proof}

\subsection{$L^2$-error estimate}

In this section we prove a sharp $L^2$-error estimate on the potential under the following usual elliptic regularity assumption:
For all $g\in L^2(\Omega)$, the unique solution $z\in H_0^1(\Omega)$ of the problem
\begin{equation}\label{eq:ell.reg.aux}
  (\GRAD z,\GRAD v) = (g,v)\qquad\forall v\in H_0^1(\Omega),
\end{equation}
satisfies
\begin{equation}\label{eq:ell.reg}
  \norm[H^2(\Omega)]{z}\le C_\Omega\norm[L^2(\Omega)]{g},
\end{equation}
with real number $C_\Omega>0$ only depending on $\Omega$.
In the proof we will need the following consistency property for the bilinear form $\mathrm{b}_h$.

\begin{proposition}[Consistency of $\mathrm{b}_h$]
  For all $\vchi\in\Hdiv$ such that $\restrto{\vchi}{T}\in\vSigma^+(T)$ for all $T\in\Th$, it holds
  \begin{equation}\label{eq:bh.cons}
    \mathrm{b}_h(\uvISh\vchi,\uv[h]) = (\DIV\vchi,v_h)\qquad\forall\uv[h]\in\uUhD.
  \end{equation}
\end{proposition}

\begin{proof}
  Recall the expression~\eqref{eq:bT.bh} of $\mathrm{b}_h$ and use commuting property~\eqref{eq:commut.DT} for $\DT$ together with the fact that $\vchi$ has continuous normal components across interfaces $F\in\Fhi$ and $v_F=0$ on all $F\in\Fhb$.
\end{proof}

\begin{theorem}[$L^2$-error estimate]\label{thm:l2.err.est}
  Let the assumptions of Theorem~\ref{thm:en.err.est} hold true, and further assume elliptic regularity, $f\in H^{k+\delta}(\Omega)$ with $\delta=1$ if $k\in\{0,1\}$ and $l=0$, $\delta=0$ otherwise.
  Then, it holds
  \begin{equation}\label{eq:L2.err.est}
    \norm{\hu-u_h}\lesssim h^{k+2}\norm[H^{k+2}(\Omega)]{u} + h^{k+2}\norm[H^{k+\delta}(\Omega)]{f}.
  \end{equation}
\end{theorem}

\begin{proof}
  Let $z$ solve~\eqref{eq:ell.reg.aux} with $g=u_h-\hu$ and set, for the sake of brevity,
  $$
  \uvhchi\eqbydef\uvISh\GRAD z,\qquad
  \uhz\eqbydef\uIUh z.
  $$
  Then, we have
  \begin{equation}\label{eq:L2.err.est:1}
    \norm{\hu-u_h}^2
    = (u-u_h,\LAPL z)
    = -(f,z) - \mathrm{b}_h(\uvhchi,\uu[h]),
  \end{equation}
  where for the first addend we have integrated by parts twice and used the fact that $-\LAPL u = f$, while for the second addend we have used the consistency property~\eqref{eq:bh.cons} of $\mathrm{b}_h$ with $\vchi=\GRAD z$ and $\uv[h]=\uu[h]$.
  Using~\eqref{eq:mixed.hyb.h:1} we get, denoting by $\vsh$ the global mixed-to-primal potential-to-flux operator whose restriction to every mesh element $T\in\Th$ coincides with $\vsT$ defined by~\eqref{eq:vsT},
  \begin{equation}\label{eq:L2.err.est:2}
    \begin{aligned}
      - \mathrm{b}_h(\uvhchi,\uu[h])
      &= \mathrm{m}_h(\uvhchi,\uvsigma[h])
      \\
      &= \mathrm{m}_h(\uvhchi - \vsh\uhz,\uvsigma[h]) + \mathrm{a}_h(\uhz,\uu[h])
      \\
      &= \mathrm{m}_h(\uvhchi - \vsh\uhz,\uvsigma[h] - \uvhsigma[h]) + \mathrm{m}_h(\uvhchi - \vsh\uhz,\uvhsigma[h]) + (f, \hz),
    \end{aligned}
  \end{equation}
  where we have inserted $\pm\vsh\uhz$ and used the fact that $\uvsigma=\vsh\uu[h]$ together with the definition~\eqref{eq:hybrid.ah} of the primal hybrid bilinear form $\mathrm{a}_h$ to pass to the second line, and we have inserted $\pm\uvhsigma[h]$ and used~\eqref{eq:hybrid.h:2} (with $\uv=\uhz$) to conclude.
  Plugging~\eqref{eq:L2.err.est:2} into~\eqref{eq:L2.err.est:1}, and observing that $(f, \hz) = (\lproj[h]{l}f,z)$ with $\lproj[h]{l}$ denoting the $L^2$-orthogonal projector on $\Uh$ (cf.~\eqref{eq:Uh}), we arrive at
  \begin{equation}\label{eq:L2.err.est:basic.est}
    \norm{\hu-u_h}^2
    = (\lproj[h]{l}f-f,z-\lproj[h]{l}z) + \mathrm{m}_h(\uvhchi - \vsh\uhz,\uvsigma[h] - \uvhsigma[h]) + \mathrm{m}_h(\uvhchi - \vsh\uhz,\uvhsigma[h]).
  \end{equation}
  Denote by $\term_1,\term_2,\term_3$ the terms in the right-hand side of~\eqref{eq:L2.err.est:basic.est}.
  For $\term_1$, if $k\in\{0,1\}$ and $l=0$, we have
  \begin{equation}\label{eq:L2.err.est:T1:1}
    |\term_1|
    \le\norm{\lproj[h]{l}f-f}\norm{z-\lproj[h]{l}z}       
    \lesssim h^2\norm[H^1(\Omega)]{f}\norm[H^1(\Omega)]{z},    
  \end{equation}
  while, in all the other cases,
  \begin{equation}\label{eq:L2.err.est:T1:2}
    |\term_1|
    \le\norm{\lproj[h]{l}f-f}\norm{z-\lproj[h]{l}z}       
    \lesssim h^{k+2}\norm[H^k(\Omega)]{f}\norm[H^2(\Omega)]{z}.
  \end{equation}

  For $\term_2$, the Cauchy--Schwarz inequality followed by (S1) and the energy error estimate~\eqref{eq:en.err.est} yields
  \begin{equation}\label{eq:L2.err.est:T2}
    |\term_2|\lesssim\norm[\vSigma,h]{\uvhchi - \vsh\uhz}\norm[\vSigma,h]{\uvsigma[h] - \uvhsigma[h]}
    \lesssim h^{k+2}\norm[H^2(\Omega)]{z}\norm[H^{k+2}(\Omega)]{u}.
  \end{equation}
  To estimate the quantity $\norm[\vSigma,h]{\uvhchi - \vsh\uhz}$ in~\eqref{eq:L2.err.est:T2}, let $\cz\in\Poly{k+1}(\Th)$ be the broken elliptic projection such that $\cz[T]\eqbydef\restrto{\cz[h]}{T}$ is defined as in~\eqref{eq:ell.proj} with $u$ replaced by $z$, observe that $\uvISh\GRADh\cz=\vsh\uIUh\cz$ by~\eqref{eq:vsT.commuting}, and use~\eqref{eq:vsT.cont} to infer
  $$
  \begin{aligned}
    \norm[\vSigma,h]{\uvhchi - \vsh\uhz}
    &\le\norm[\vSigma,h]{\uvISh(\GRAD z - \GRADh\cz)} + \norm[\vSigma,h]{\vsh\uIUh(z-\cz)}
    \\
    &\lesssim\norm[\vSigma,h]{\uvIST(\GRAD z - \GRADh\cz)} + \norm[U,h]{\uIUh(z-\cz)}
    \lesssim h\norm[H^2(\Omega)]{z},
  \end{aligned}
  $$
  where the conclusion follows from the stability of the $L^2$-projector and the optimal approximation properties of $\cz$.

  For $\term_3$, recalling the definitions~\eqref{eq:mh} of $\mathrm{m}_h$,~\eqref{eq:mT} of $\mathrm{m}_T$, and (S2), we have
  $$
  \begin{aligned}
    \term_3
    &= \sum_{T\in\Th}(\ST(\uvhchi-\vsT\uhz[T]),\ST\uvhsigma)_T
    \\
    &= \sum_{T\in\Th}(\PT\uvhchi-\GRAD\cz[T],\GRAD\cu)_T
    \\
    &= \sum_{T\in\Th}\left(
    (\GRAD(z-\cz[T]),\GRAD\cu[T])_T + \sum_{F\in\Fh[T]}(\lproj[F]{k}(\GRAD z\SCAL\normal_{TF})-\GRAD z\SCAL\normal_{TF},\cu[T])_F
    \right)
    \\
    &= \sum_{T\in\Th}\sum_{F\in\Fh[T]}(\lproj[F]{k}(\GRAD z\SCAL\normal_{TF})-\GRAD z\SCAL\normal_{TF},\cu[T]-u)_F,
  \end{aligned}
  $$
  where we have used the definition~\eqref{eq:ST} of $\ST$ together with the orthogonal decomposition~\eqref{eq:decomp.PTs} and the fact that $(\ST\circ\vsT)\uhz[T]=\GT\uhz[T]=\GRAD\cz[T]$ (cf.~\eqref{eq:char.GT} and~\eqref{eq:commut.GT}) to pass to the second line, the definition~\eqref{eq:PT} of $\PT$ (with $\uvtau=\uvhchi$ and $w=\cu[T]$) together with the fact that $\DT\uvhchi=\LAPL z$ and an integration by parts to pass to the third line, and concluded in the fourth line using the fact that $\cz[T]$ is a local elliptic projection to cancel the first term together with the fact that the quantity $(\lproj[F]{k}(\GRAD z\SCAL\normal_{TF})-\GRAD z\SCAL\normal_{TF})$ is single-valued on every interface $F\in\Fhi$ and $u=0$ on all $F\in\Fhb$ to insert $u$ into the second term.
  
  Using the Cauchy--Schwarz inequality and the optimal approximation properties of $\lproj[F]{k}$ and $\cu[T]$, we conclude
  \begin{equation}\label{eq:L2.err.est:T3}
    |\term_3|\lesssim h^{k+2}\norm[H^{k+2}(\Omega)]{u}\norm[H^2(\Omega)]{z}.
  \end{equation}
  Using~\eqref{eq:L2.err.est:T1:1}--\eqref{eq:L2.err.est:T3} to estimate the right-hand side of~\eqref{eq:L2.err.est:basic.est} followed by the elliptic regularity~\eqref{eq:ell.reg} to bound $\norm[H^2(\Omega)]{z}\lesssim\norm{\hu-u_h}$, the desired result follows.
\end{proof}


\begin{small}
  \bibliographystyle{plain}
  \bibliography{umho}

\begin{thebibliography}{10}

\bibitem{Aghili.Boyaval.ea:15}
J.~Aghili, S.~Boyaval, and D.~A. Di~Pietro.
\newblock Hybridization of mixed high-order methods on general meshes and
  application to the {Stokes} equations.
\newblock {\em Comput. Meth. Appl. Math.}, 15(2):111--134, 2015.

\bibitem{Antonietti.Giani.ea:13}
P.~F. Antonietti, S.~Giani, and P.~Houston.
\newblock $hp$-version composite discontinuous {G}alerkin methods for elliptic
  problems on complicated domains.
\newblock {\em SIAM J. Sci. Comput.}, 35(3):A1417--A1439, 2013.

\bibitem{Araya.Harder.ea:13}
R.~Araya, C.~Harder, D.~Paredes, and F.~Valentin.
\newblock Multiscale hybrid-mixed method.
\newblock {\em SIAM J. Numer. Anal.}, 51(6):3505--3531, 2013.

\bibitem{Arbogast.Chen:95}
T.~Arbogast and Z.~Chen.
\newblock On the implementation of mixed methods as nonconforming methods for
  second-order elliptic problems.
\newblock {\em Math. Comp.}, 64:943--972, 1995.

\bibitem{Arnold.Brezzi:85}
D.~N. Arnold and F.~Brezzi.
\newblock Mixed and nonconforming finite element methods: implementation,
  postprocessing and error estimates.
\newblock {\em RAIRO Mod\'el. Math. Anal. Num.}, 19(4):7--32, 1985.

\bibitem{Ayuso-de-Dios.Lipnikov.ea:15}
B.~Ayuso~de Dios, K.~Lipnikov, and G.~Manzini.
\newblock The nonconforming virtual element method.
\newblock {\em ESAIM: Math. Model Numer. Anal. (M2AN)}, 50(3):879--904, 2016.

\bibitem{Bahriawati.Carstensen:05}
C.~Bahriawati and C.~Carstensen.
\newblock Three {Matlab} implementations of the lowest-order {Raviart--Thomas
  MFEM} with a posteriori error control.
\newblock {\em Comput. Meth. Appl. Math.}, 5(4):333--361, 2005.

\bibitem{Bassi.Botti.ea:12}
F.~Bassi, L.~Botti, A.~Colombo, D.~A. Di~Pietro, and P.~Tesini.
\newblock On the flexibility of agglomeration based physical space
  discontinuous {Galerkin} discretizations.
\newblock {\em J. Comput. Phys.}, 231(1):45--65, 2012.

\bibitem{Beirao-da-Veiga.Brezzi.ea:13}
L.~Beir\~{a}o~da Veiga, F.~Brezzi, A.~Cangiani, G.~Manzini, L.~D. Marini, and
  A.~Russo.
\newblock Basic principles of virtual element methods.
\newblock {\em Math. Models Methods Appl. Sci. (M3AS)}, 199(23):199--214, 2013.

\bibitem{Beirao-da-Veiga.Brezzi.ea:13*1}
L.~Beir\~{a}o~da Veiga, F.~Brezzi, and L.~D. Marini.
\newblock Virtual elements for linear elasticity problems.
\newblock {\em SIAM J. Numer. Anal.}, 2(51):794--812, 2013.

\bibitem{Beirao-da-Veiga.Brezzi.ea:15}
L.~Beir\~{a}o~da Veiga, F.~Brezzi, L.~D. Marini, and A.~Russo.
\newblock {$H(\mathrm{div})$} and {$H(\mathrm{curl})$}-conforming {VEM}.
\newblock {\em Numer. Math.}, 133:303--332, 2016.

\bibitem{Beirao-da-Veiga.Brezzi.ea:16}
L.~Beir\~{a}o~da Veiga, F.~Brezzi, L.~D. Marini, and A.~Russo.
\newblock Mixed virtual element methods for general second order elliptic
  problems on polygonal meshes.
\newblock {\em ESAIM: Math. Model. Numer. Anal. (M2AN)}, 50(3):727--747, 2016.

\bibitem{Beirao-da-Veiga.Lipnikov.ea:14}
L.~Beir\~{a}o~da Veiga, K.~Lipnikov, and G.~Manzini.
\newblock {\em The Mimetic Finite Difference Method for Elliptic Problems},
  volume~11 of {\em Modeling, Simulation and Applications}.
\newblock Springer, 2014.

\bibitem{Boffi.Brezzi.ea:13}
D.~Boffi, F.~Brezzi, and M.~Fortin.
\newblock {\em Mixed finite element methods and applications}, volume~44 of
  {\em Springer Series in Computational Mathematics}.
\newblock Springer, Heidelberg, 2013.

\bibitem{Bonelle.Ern:14}
J.~Bonelle and A.~Ern.
\newblock Analysis of compatible discrete operator schemes for elliptic
  problems on polyhedral meshes.
\newblock {\em ESAIM: Math. Model. Numer. Anal. (M2AN)}, 48:553--581, 2014.

\bibitem{Brezzi.Buffa.ea:09}
F.~Brezzi, A.~Buffa, and K.~Lipnikov.
\newblock Mimetic finite difference for elliptic problem.
\newblock {\em ESAIM: Math. Model. Numer. Anal. (M2AN)}, 43:277--295, 2009.

\bibitem{Brezzi.Falk.ea:14}
F.~Brezzi, R.~S. Falk, and L.~D. Marini.
\newblock Basic principles of mixed virtual element methods.
\newblock {\em ESAIM Math. Model. Numer. Anal. (M2AN)}, 48(4):1227--1240, 2014.

\bibitem{Brezzi.Lipnikov.ea:05}
F.~Brezzi, K.~Lipnikov, and M.~Shashkov.
\newblock Convergence of the mimetic finite difference method for diffusion
  problems on polyhedral meshes.
\newblock {\em SIAM J. Numer. Anal.}, 43(5):1872--1896, 2005.

\bibitem{Cangiani.Georgoulis.ea:14}
A.~Cangiani, E.~H. Georgoulis, and P.~Houston.
\newblock {$hp$}-version discontinuous {G}alerkin methods on polygonal and
  polyhedral meshes.
\newblock {\em Math. Models Methods Appl. Sci.}, 24(10):2009--2041, 2014.

\bibitem{Castillo.Cockburn.ea:00}
P.~Castillo, B.~Cockburn, I.~Perugia, and D.~Sch\"{o}tzau.
\newblock An a priori error analysis of the local discontinuous {Galerkin}
  method for elliptic problems.
\newblock {\em SIAM J. Numer. Anal.}, 38:1676--1706, 2000.

\bibitem{Chen:96}
Z.~Chen.
\newblock Equivalence between and multigrid algorithms for nonconforming and
  mixed methods for second-order elliptic problems.
\newblock {\em East-West J. Numer. Math.}, 4:1--33, 1996.

\bibitem{Cockburn.Di-Pietro.ea:15}
B.~Cockburn, D.~A. Di~Pietro, and A.~Ern.
\newblock Bridging the {Hybrid High-Order} and {Hybridizable Discontinuous
  Galerkin} methods.
\newblock {\em ESAIM: Math. Model. Numer. Anal. (M2AN)}, 50(3):635--650, 2016.

\bibitem{Cockburn.Gopalakrishnan.ea:09}
B.~Cockburn, J.~Gopalakrishnan, and R.~Lazarov.
\newblock Unified hybridization of discontinuous {G}alerkin, mixed, and
  continuous {G}alerkin methods for second order elliptic problems.
\newblock {\em SIAM J. Numer. Anal.}, 47(2):1319--1365, 2009.

\bibitem{Codecasa.Specogna.ea:10}
L.~Codecasa, R.~Specogna, and F.~Trevisan.
\newblock A new set of basis functions for the discrete geometric approach.
\newblock {\em J. Comput. Phys.}, 19(299):7401--7410, 2010.

\bibitem{Crouzeix.Raviart:73}
M.~Crouzeix and P.-A. Raviart.
\newblock Conforming and nonconforming finite element methods for solving the
  stationary {S}tokes equations.
\newblock {\em RAIRO Mod\'el. Math. Anal. Num.}, 7(3):33--75, 1973.

\bibitem{Di-Pietro:12}
D.~A. Di~Pietro.
\newblock Cell centered {Galerkin} methods for diffusive problems.
\newblock {\em ESAIM: Math. Model. Numer. Anal. (M2AN)}, 46(1):111--144, 2012.

\bibitem{Di-Pietro:13}
D.~A. Di~Pietro.
\newblock On the conservativity of cell centered {Galerkin} methods.
\newblock {\em C. R. Acad. Sci Paris, Ser. I}, 351:155--159, 2013.

\bibitem{Di-Pietro.Droniou:15}
D.~A. Di~Pietro and J.~Droniou.
\newblock A {Hybrid High-Order} method for {Leray--Lions} elliptic equations on
  general meshes.
\newblock {\em Math. Comp.}
\newblock Accepted for publication.
  Preprint~\href{http://arxiv.org/abs/1508.01918}{arXiv:1508.01918} [math.NA].

\bibitem{Di-Pietro.Droniou:16}
D.~A. Di~Pietro and J.~Droniou.
\newblock {$W^{s,p}$}-approximation properties of elliptic projectors on
  polynomial spaces, with application to the error analysis of a {Hybrid
  High-Order} discretisation of {Leray--Lions} problems.
\newblock 2016.
\newblock Submitted.
  Preprint~\href{http://arxiv.org/abs/1606.02832}{arXiv:1606.02832} [math.NA].

\bibitem{Di-Pietro.Droniou.ea:15}
D.~A. Di~Pietro, J.~Droniou, and A.~Ern.
\newblock A discontinuous-skeletal method for advection-diffusion-reaction on
  general meshes.
\newblock {\em SIAM J. Numer. Anal.}, 53(5):2135--2157, 2015.

\bibitem{Di-Pietro.Ern:12}
D.~A. Di~Pietro and A.~Ern.
\newblock {\em Mathematical aspects of discontinuous {G}alerkin methods},
  volume~69 of {\em Math\'ematiques \& Applications}.
\newblock Springer-Verlag, Berlin, 2012.

\bibitem{Di-Pietro.Ern:15*1}
D.~A. Di~Pietro and A.~Ern.
\newblock A hybrid high-order locking-free method for linear elasticity on
  general meshes.
\newblock {\em Comput. Meth. Appl. Mech. Engrg.}, 283:1--21, 2015.

\bibitem{Di-Pietro.Ern:16}
D.~A. Di~Pietro and A.~Ern.
\newblock Arbitrary-order mixed methods for heterogeneous anisotropic diffusion
  on general meshes.
\newblock {\em IMA J. Numer. Anal.}, 2016.
\newblock Published online.
  DOI~\href{http://dx.doi.org/10.1093/imanum/drw003}{10.1093/imanum/drw003}.

\bibitem{Di-Pietro.Ern.ea:14}
D.~A. Di~Pietro, A.~Ern, and S.~Lemaire.
\newblock An arbitrary-order and compact-stencil discretization of diffusion on
  general meshes based on local reconstruction operators.
\newblock {\em Comput. Meth. Appl. Math.}, 14(4):461--472, 2014.

\bibitem{Droniou.Eymard:06}
J.~Droniou and R.~Eymard.
\newblock A mixed finite volume scheme for anisotropic diffusion problems on
  any grid.
\newblock {\em Numer. Math.}, 105:35--71, 2006.

\bibitem{Droniou.Eymard.ea:10}
J.~Droniou, R.~Eymard, T.~Gallou\"{e}t, and R.~Herbin.
\newblock A unified approach to mimetic finite difference, hybrid finite volume
  and mixed finite volume methods.
\newblock {\em Math. Models Methods Appl. Sci. (M3AS)}, 20(2):1--31, 2010.

\bibitem{Droniou.Eymard.ea:13}
J.~Droniou, R.~Eymard, T.~Gallouet, and R.~Herbin.
\newblock Gradient schemes: a generic framework for the discretisation of
  linear, nonlinear and nonlocal elliptic and parabolic equations.
\newblock {\em Math. Models Methods Appl. Sci. (M3AS)}, 23(13):2395--2432,
  2013.

\bibitem{Dupont.Scott:80}
T.~Dupont and R.~Scott.
\newblock Polynomial approximation of functions in {S}obolev spaces.
\newblock {\em Math. Comp.}, 34(150):441--463, 1980.

\bibitem{Eymard.Gallouet.ea:10}
R.~Eymard, T.~Gallou{\"e}t, and R.~Herbin.
\newblock Discretization of heterogeneous and anisotropic diffusion problems on
  general nonconforming meshes. {SUSHI}: a scheme using stabilization and
  hybrid interfaces.
\newblock {\em IMA J. Numer. Anal.}, 30(4):1009--1043, 2010.

\bibitem{Eymard.Guichard.ea:12}
R.~Eymard, C.~Guichard, and R.~Herbin.
\newblock Small-stencil 3{D} schemes for diffusive flows in porous media.
\newblock {\em ESAIM Math. Model. Numer. Anal.}, 46(2):265--290, 2012.

\bibitem{Lehrenfeld:10}
C.~Lehrenfeld.
\newblock {\em Hybrid Discontinuous Galerkin methods for solving incompressible
  flow problems}.
\newblock PhD thesis, Rheinisch-Westf\"alischen Technischen Hochschule Aachen,
  2010.

\bibitem{Lipnikov.Manzini:14}
K.~Lipnikov and G.~Manzini.
\newblock A high-order mimetic method on unstructured polyhedral meshes for the
  diffusion equation.
\newblock {\em J. Comput. Phys.}, 272:360--385, 2014.

\bibitem{Marini:85}
L.D. Marini.
\newblock An inexpensive method for the evaluation of the solution of the
  lowest order {Raviart--Thomas} mixed method.
\newblock {\em SIAM J. Numer. Anal.}, 22:493--496, 1985.

\bibitem{Raviart.Thomas:77}
P.~A. Raviart and J.~M. Thomas.
\newblock A mixed finite element method for 2nd order elliptic problems.
\newblock In I.~Galligani and E.~Magenes, editors, {\em Mathematical Aspects of
  the Finite Element Method}. Springer, New York, 1977.

\bibitem{Tonti:75}
E.~Tonti.
\newblock {\em On the formal structure of physical theories}.
\newblock Istituto di Matematica del Politecnico di Milano, 1975.

\bibitem{Vohralik.Wohlmuth:13}
M.~Vohral{\'{\i}}k and B.~I. Wohlmuth.
\newblock Mixed finite element methods: implementation with one unknown per
  element, local flux expressions, positivity, polygonal meshes, and relations
  to other methods.
\newblock {\em Math. Models Methods Appl. Sci. (M3AS)}, 23(5):803--838, 2013.

\end{thebibliography}
\end{small}

\end{document}